\documentclass[12pt,a4paper]{amsart}
\usepackage{amsmath, amssymb, amsfonts, float}
\usepackage[all]{xy}
\usepackage{caption}
\usepackage{enumerate}
\usepackage{comment}
\usepackage{bm}
\usepackage{graphicx}
\usepackage[breaklinks=true]{hyperref}
\usepackage{mathpazo}
\usepackage{xcolor}
\usepackage{multicol}
\usepackage{tabularx}
\usepackage{a4wide}
\usepackage[capitalize]{cleveref}

\usepackage{hyperref}

\usepackage[normalem]{ulem}

\theoremstyle{plain}
\newtheorem{thm}{Theorem}[section]

\newtheorem{lem}[thm]{Lemma}
\newtheorem{cor}[thm]{Corollary}
\newtheorem{prop}[thm]{Proposition}

\newtheorem{conj}[thm]{Conjecture}

\theoremstyle{definition}
\newtheorem{definition}[thm]{Definition}
\newtheorem{ex}[thm]{Example}

\newtheorem{question}[thm]{Question}

\newtheorem{rmk}[thm]{Remark}

\makeatletter
\newcommand{\neutralize}[1]{\expandafter\let\csname c@#1\endcsname\count@}
\makeatother

\renewcommand{\baselinestretch}{1.25}
\newcommand{\doublespaced}{\renewcommand{\baselinestretch}{2}\normalfont}
\newcommand{\singlespaced}{\renewcommand{\baselinestretch}{1}\normalfont}
\newcommand{\draftspaced}{\doublespaced} 

\newcommand{\uh}{\underline{\underline{H}}}
\newcommand{\ldr}{\langle D \rangle}
\newcommand{\al}{\alpha}
\newcommand{\p}{\partial}
\newcommand{\e}{\epsilon}
\newcommand{\eq}[2]{\begin{equation}\label{#1}#2 \end{equation}}
\newcommand{\ml}[2]{\begin{multline}\label{#1}#2 \end{multline}}
\newcommand{\ga}[2]{\begin{gather}\label{#1}#2 \end{gather}}
\newcommand{\gr}{{\rm gr}}
\newcommand{\mc}{\mathcal}
\newcommand{\mb}{\mathbb}
\newcommand{\surj}{\twoheadrightarrow}
\newcommand{\inj}{\hookrightarrow}

\newcommand{\red}{{\rm red}}
\newcommand{\cd}{{\rm cd}}
\newcommand{\ed}{{\rm ed}}
\newcommand{\codim}{{\rm codim}}
\newcommand{\rank}{{\rm rank}}
\newcommand{\Pic}{{\rm Pic}}
\newcommand{\Quot}{{\rm Quot}}
\newcommand{\Div}{{\rm Div}}
\newcommand{\Hom}{{\rm Hom}}
\newcommand{\im}{{\rm im}}
\newcommand{\Spec}{{\rm Spec \,}}
\newcommand{\Sing}{{\rm Sing}}
\newcommand{\Char}{{\rm char}}
\newcommand{\Tr}{{\rm Tr}}
\newcommand{\trdeg}{{\rm trdeg}}
\newcommand{\Gal}{{\rm Gal}}
\newcommand{\Min}{{\rm Min \ }}
\newcommand{\Max}{{\rm Max \ }}
\newcommand{\rt}{{\rm res \ Tr}}
\newcommand{\trace}{{\rm Tr}}
\newcommand{\sym}{\text{Sym}}
\newcommand{\sA}{{\mathcal A}}
\newcommand{\sB}{{\mathcal B}}
\newcommand{\sC}{{\mathcal C}}
\newcommand{\sD}{{\mathcal D}}
\newcommand{\sE}{{\mathcal E}}
\newcommand{\sF}{{\mathcal F}}
\newcommand{\sG}{{\mathcal G}}
\newcommand{\sH}{{\mathcal H}}
\newcommand{\sI}{{\mathcal I}}
\newcommand{\sJ}{{\mathcal J}}
\newcommand{\sK}{{\mathcal K}}
\newcommand{\sL}{{\mathcal L}}
\newcommand{\sM}{{\mathcal M}}
\newcommand{\sN}{{\mathcal N}}
\newcommand{\sO}{{\mathcal O}}
\newcommand{\sP}{{\mathcal P}}
\newcommand{\sQ}{{\mathcal Q}}
\newcommand{\sR}{{\mathcal R}}
\newcommand{\sS}{{\mathcal S}}
\newcommand{\sT}{{\mathcal T}}
\newcommand{\sU}{{\mathcal U}}
\newcommand{\sV}{{\mathcal V}}
\newcommand{\sW}{{\mathcal W}}
\newcommand{\sX}{{\mathcal X}}
\newcommand{\sY}{{\mathcal Y}}
\newcommand{\sZ}{{\mathcal Z}}
\newcommand{\A}{{\mathbb A}}
\newcommand{\B}{{\mathbb B}}
\newcommand{\C}{{\mathbb C}}
\newcommand{\D}{{\mathbb D}}
\newcommand{\E}{{\mathbb E}}
\newcommand{\F}{{\mathbb F}}
\newcommand{\G}{{\mathbb G}}
\renewcommand{\H}{{\mathbb H}}
\newcommand{\J}{{\mathbb J}}
\newcommand{\M}{{\mathbb M}}
\newcommand{\N}{{\mathbb N}}
\renewcommand{\P}{{\mathbb P}}
\newcommand{\Q}{{\mathbb Q}}
\newcommand{\R}{{\mathbb R}}
\newcommand{\T}{{\mathbb T}}
\newcommand{\U}{{\mathbb U}}
\newcommand{\V}{{\mathbb V}}
\newcommand{\W}{{\mathbb W}}
\newcommand{\X}{{\mathbb X}}
\newcommand{\OO}{{\mathcal O}}

\newcommand{\Y}{{\mathbb Y}}
\newcommand{\Z}{{\mathbb Z}}
\newcommand{\mh}{{\bf mh \ }}
\newcommand{\MIC}{\mbox{MIC}}
\newcommand{\Mod}{\text{\sf Mod}}
\newcommand{\vect}{\text{\sf vect}}
\newcommand{\Vect}{\text{\sf Vect}}
\newcommand{\Rep}{\text{\sf Rep}}
\newcommand{\id}{{\rm id\hspace{.1ex}}}
\newcommand{\Ind}{\text{\rm Ind-}}
\newcommand{\ord}{\text{\rm ord}}
\newcommand{\mFe}{{\tilde{F}_{\Delta}}}
\newcommand{\Fe}{{\tilde{F}_{\Delta}}}
\newcommand{\rcite}{\textcolor{red}{cite}}

\def\Sp{\Spec}
\newcommand{\res}{{\rm res \hspace{.1ex} }}
\newcommand{\ind}{{\text{\sf ind}\hspace{.1ex}}}
\newcommand{\tung}[1]{{\color{blue}#1}}

\begin{document}
\title{On the arithmetic of \\ generalized Fekete polynomials }
\date{}
\dedicatory{Dedicated to Professor Kazuya Kato on the occasion of his 70th birthday }

 \author{ J\'an Min\'a\v{c}, Tung T. Nguyen, Nguy$\tilde{\text{\^{e}}}$n Duy T\^{a}n }
\address{Department of Mathematics, Western University, London, Ontario, Canada N6A 5B7}
\email{minac@uwo.ca}
\date{\today}

 \address{Department of Mathematics, Western University, London, Ontario, Canada N6A 5B7}
 \email{tungnt@uchicago.edu}
 
  \address{
 School of Applied  Mathematics and 	Informatics, Hanoi University of Science and Technology, 1 Dai Co Viet Road, Hanoi, Vietnam } 
\email{tan.nguyenduy@hust.edu.vn}

\thanks{JM is partially supported by the Natural Sciences and Engineering Research Council of Canada (NSERC) grant R0370A01. He gratefully acknowledges the Western University Faculty of Science Distinguished Professorship 2020-2021. NDT is funded by Vingroup Joint Stock Company and supported by Vingroup Innovation Foundation (VinIF) under the project code VINIF.2021.DA00030
}
\begin{abstract}
For each prime number $p$ one can associate a Fekete polynomial with coefficients $-1$ or $1$ except the constant term, which is 0. These are classical polynomials that have been studied extensively in the framework of analytic number theory. In a recent paper, we showed that these polynomials also encode interesting arithmetic information. In this paper, we define generalized Fekete polynomials associated with quadratic characters whose conductors could be composite numbers. We then investigate the appearance of cyclotomic factors in these generalized Fekete polynomials. Based on this investigation, we introduce a compact version of Fekete polynomials as well as their trace polynomials. We then study the Galois groups of these Fekete polynomials using modular techniques. In particular,  we discover some surprising extra symmetries which imply some restrictions on the corresponding Galois groups. Finally, based on both theoretical and numerical data, we propose a precise conjecture on the structure of these Galois groups. 

\end{abstract}

\maketitle

\tableofcontents
\section{Introduction} 

Fekete polynomials are classical polynomials whose history can be traced to the 19th century in relation to the studies of Dirichlet $L$-functions. In fact, as explained very nicely in \cite[Chapter 10, Section 3]{[Lemmermeyer1]} these polynomials already played a significant role in Gauss's original sixth proof of the quadratic reciprocity law. 

We recall that for each prime $p$, the Fekete polynomial $F_p(x)$ associated with $p$ is defined by 
\[ F_p(x)=\sum_{a=1}^{p-1} \left(\frac{a}{p} \right) x^a .\] 
Here $\left(\dfrac{a}{p}\right)$ is the Legendre symbol which is equal to, for $1\leq a\leq p-1$, the value $1$ if $a$ is a quadratic residue modulo $p$ and the value -1 if $a$ is not a quadratic residue modulo $p$.

In our previous work \cite{[MTT3]}, we have studied the arithmetic of these Fekete polynomials. We also described briefly some early history of Fekete polynomials and their role in the studies of some $L$-functions. Additionally, we also introduced their cousin which we called reduced Fekete polynomials $g_p(x)$. Our main theorems in \cite{[MTT3]} show that these polynomials contain interesting arithmetic information such as the class numbers or the orders of certain $K$-groups. Furthermore, our experimental data suggests that the Galois group of $f_{p}(x)$ and $g_p(x)$ are as large as possible (see \cite[Conjecture 4.9, Conjecture 4.13]{[MTT3]}).

F. Lemmermeyer in \cite[Page 231]{[Lemmermeyer1]} naturally extended Fekete polynomials for general Dirichlet characters $\chi.$ When $\chi$ is a primitive quadratic Dirichlet character with conductor equals to a prime number $p$, Lemmermeyer's definition specializes to $F_p(x)$ that we have recalled above. Note that when $p=2$, $F_2(x)=x$ and hence we can restrict our attention to $p$ odd. We remark also that these Fekete polynomials also appear implicitly in the work of Baker and Montgomery about quadratic $L$-functions (see \cite[Page 24]{[BM]}). 

In this paper, we adopt Lemmermeyer's approach. We shall call these polynomials generalized Fekete polynomials $F_{\Delta}(x)$. These are polynomials associated with a (quadratic) character $\chi=\chi_{\Delta}$ where the conductor $D=|\Delta|$ of $\chi$ is not necessarily a prime number (see \ref{def:fekete} for the precise definition). We will study the factorization of $F_{\Delta}(x)$, which leads to the definitions of the generalized Fekete polynomial $f_{\Delta}(x)$ and their reduced version $g_{\Delta}(x)$. In this present article, we will mostly focus on the precise determination of $f_{\Delta}(x)$ and $g_{\Delta}(x)$ and their Galois groups. Further arithmetical properties of $f_{\Delta}(x)$ and $g_{\Delta}(x)$ will be discussed in a separate paper in preparation. We shall see that studying these polynomials presents a rich opportunity to extend considerably arithmetic and Galois theoretic considerations of a new class of polynomials with interesting properties. We make considerable effort to make our article as self-contained as possible and we provide all key definitions and motivations. Some of our proofs are involved and intricate and we attempt to explain them with all details in a lucid way. We hope that our article opens a new interesting area of number theory and related areas. In particular, our investigation in the last section about the conjectural Galois groups of Fekete's polynomials suggests that it would be interesting to study the irreducibility of these polynomials using their special values. For some related references regarding this approach see \cite{[irr1], [irr2], [irr3], [irr4]}.

The structure of this article is as follows. In \cref{sec:background}, we recall some basic definitions in algebraic number theory including the Legendre, Jacobi, and Kronecker symbols and their relations with quadratic fields. We will also prove some lemmas about $\Q$-bases of cyclotomic extensions. In \cref{sec:fekete}, we introduce the generalized Fekete polynomial $F_{\Delta}(x)$ and their modified version when the conductor is even. Using the theory of Gauss sums, we discuss some cyclotomic factors of the generalized Fekete polynomials. \cref{sec:multiplicy_Phi1_Phi2} will focus on the multiplicity of $x=1$ and $x=-1$ in $F_{\Delta}(x).$ \cref{sec:general_n} and \cref{sec:n_not_divisor} will discuss some partial results about the appearance of cyclotomic factors in $F_{\Delta}(x).$ We will then focus exclusively on the case where $\Delta$ has a few prime factors. More specifically, we will discuss the case $\Delta \in \{4p, -4p, 3p, -3p \}$ in Section 7-10. This is done by an extensive study of cyclotomic factors of $F_{\Delta}(x).$ The calculations in these sections lead to the precise definition of the Fekete polynomial $f_{\Delta}(x)$ and its trace polynomial $g_{\Delta}(x)$. In the final section, we provide some numerical data for the Galois groups of $f_{\Delta}(x)$ and $g_{\Delta}(x).$ In particular, we discuss some cases where a surprising extra symmetry appears, hence putting some restrictions on the Galois group of $f_{\Delta}(x)$. Additionally, we discuss some tests to detect the Galois group of a reciprocal polynomial, which might be of independent interest. Finally, we propose a precise conjecture on the structure of the Galois groups of $f_{\Delta}(x)$ and $g_{\Delta}(x)$ based on our experimental data. 

\section{Background in algebraic number theory} 
\label{sec:background}
\subsection{Legendre symbol, Jacobi symbol, and Kronecker symbol} 
In this section, we recall the definition and basic properties of the Legendre, Jacobi, and Kronecker symbols. Our main references for this section are \cite{[LMFDB]} and \cite[Chapter 9]{[MV]}. 

We first recall the definition of the Legendre symbol. Let $a$ be an integer and $p$ a prime number. The Legendre symbol $\left(\dfrac{a}{p} \right)$ is defined as follows. 
\[ \left(\frac{a}{p} \right)  = \begin{cases}
   0 & \text{if $p|a$ } \\ 
  1  &  \text{if $a$ is a square modulo $p$ } \\
  -1 &  \text{else. }
\end{cases} \]  
The Legendre symbol is multiplicative. In fact, it is a Dirichlet character with conductor $p$. The Jacobi symbol $\left(\frac{a}{b} \right)$, where $b$ is an odd positive integer, is a generalization of the Legendre symbol. Specifically, suppose that $b$ has the following prime factorization $ b= p_1^{e_1} p_2^{e_2} \ldots p_r^{e_r}.$  Then the Jacobi symbol $\left(\dfrac{a}{b} \right)$ is defined as follow 
\[ \left(\frac{a}{b} \right) = \left(\frac{a}{p_1} \right)^{e_1} \left(\frac{a}{p_2} \right)^{e_2} \ldots \left(\frac{a}{p_r} \right)^{e_r} , \]
where $\left(\dfrac{a}{p_i} \right)$ is the Legendre symbol. Finally, we recall the Kronecker symbol, which generalizes both the Legendre and the Jacobi symbols.  To do so, we first define the following conventions
\[ \left(\frac{a}{-1} \right) = \begin{cases}
  1  &  \text{if } a \geq 0 \\
  -1 &  \text{if } a<0
\end{cases}, \; 
\left(\frac{a}{2} \right) = \begin{cases}
  0  &  \text{if $2|a$ }  \\
  1 &  \text{if } a \equiv \pm{1} \pmod{8} \\ 
  -1 & \text{if } a \equiv \pm{3} \pmod{8}
\end{cases}, \;
 \left(\frac{a}{0} \right) = \begin{cases}
  1  &  \text{if } a = \pm{1} \\
  0 &  \text{otherwise. }
\end{cases} \] 
Let $a, n$ be integers. Suppose that $n$ has the following factorization into the product of distinct prime numbers 
\[ n = \text{sgn}(n) p_1^{e_1} p_2^{e_2} \ldots p_r^{e_r} .\] 
Here $\text{sgn}(n)$ is the sign of $n$, which is $1$ if $n>0$ and $-1$ otherwise. Then, the Kronecker symbol $\left(\frac{a}{n} \right)$ is defined as 
\[ \left(\frac{a}{n} \right) = \left(\frac{a}{\text{sgn}(n)} \right) \left(\frac{a}{p_1} \right)^{e_1} \left(\frac{a}{p_2} \right)^{e_2} \ldots \left(\frac{a}{p_r} \right)^{e_r} .\]   
We note that the Kronecker symbol is defined for all integers $a,n$. 

\subsection{Quadratic characters} 
Let $d$ be a squarefree integer. Let $\Delta$ be the discriminant of the quadratic extension $\Q(\sqrt{d})/\Q$, which is given by 
\[ \Delta = \begin{cases}
  d  &  \text{if } d \equiv 1 \pmod{4} \\
  4d &  \text{if } d \equiv 2, 3 \pmod{4}.
\end{cases} \] 
From this definition, we can see that $\Delta$ necessarily determines the quadratic extension $\Q(\sqrt{d})/\Q.$ We have the following theorem. 
\begin{thm} (\cite[Theorem 9.13]{[MV]})
Let $\chi_{\Delta}: \Z \to \mathbb{C}^{\times}$ be the function given by 
\[ \chi_{\Delta}(a) = \left(\frac{\Delta}{a} \right) ,\] 
where $\left(\frac{\Delta}{a} \right)$ is the Kronecker symbol introduced in the previous section. Then $\chi_{\Delta}$ is a primitive quadratic character of conductor $D=|\Delta|.$ Furthermore, every primitive quadratic character is given uniquely this way. 
\end{thm}

\subsection{Cyclotomic polynomials}
We recall some properties of cyclotomic polynomials that we will use throughout this article. For a more detailed treatment, interested readers are encouraged to see \cite[Sections 46 and 48]{[Nagell]}. For a more recent survey on both coefficients and higher-order derivatives of cyclotomic polynomials, we recommend \cite{[cyclotomic]}.

Let $n$ be a positive integer. The $n$-th cyclotomic polynomial $\Phi_n(x.)$ is defined as 
\[ \Phi_n(x) = \prod_{\zeta}(x-\zeta) ,\]
where $\zeta$ runs over the set of all primitive $n$-roots of unity. It is known that $\Phi_n(x)$ is a polynomial of degree $\varphi(n)$ where $\varphi(n)$ is the Euler totient function. Additionally, we have the following factorization 
\[ x^n-1 =\prod_{d|n} \Phi_d(x). \] 
We have the following identities. 
\begin{prop}(\cite[Section 46]{[Nagell]}) \label{prop:cyclotomic}

Let $n$ be a positive integer and $p$ a prime number. We have the following statements. 
\begin{enumerate}
    \item If $\gcd(p,n)=1$, then 
    $\Phi_{np}(x) \Phi_n(x)= \Phi_n(x^p).$
    \item If $p|n$ then $\Phi_{np}(x)=\Phi_n(x^p).$
    \item If $n$ is odd and $n>1$ then  $\Phi_{2n}(x)= \Phi_n(-x),$ and $\Phi_{4n}(x)=\Phi_{2n}(x^2)=\Phi_n(-x^2).$
    \item Additionally 
    $\Phi_4(x)=x^2+1 = \Phi_2(x^2) =- \Phi_1(-x^2).$
\end{enumerate}
\end{prop}
Here are some examples of cyclotomic polynomials. 
\begin{ex} \label{ex:cyclotomic_polynomials}
$ \Phi_1(x)=x-1, \Phi_2(x)=x+1.$
If $p$ is an odd prime number then 
\[ \Phi_p(x)=\frac{x^p-1}{x-1} = \sum_{k=0}^{p-1} x^k,
\Phi_{2p}(x)=\Phi_p(-x),
\Phi_{4p}(x)=\Phi_p(-x^2),
\] 
\end{ex}

We have the following observation. 
\begin{lem} \label{lem:reciprocal} 
$\Phi_1(x)^2$ and $\Phi_n(x)$, for $n \geq 2$,   are reciprocal polynomials. 
\end{lem}

\begin{proof}
For the first statement, we note that $\Phi_1(x)^2 = (x-1)^2=x^2-2x+1$. This is clearly a reciprocal polynomial.  If $n=2$ then 
$\Phi_2(x)=x+1$. This is also clearly a reciprocal polynomial. Let us assume now that $n >2$. By definition, we have 
$ \Phi_n (x)= \prod_{\zeta}(x-\zeta),$ 
where $\zeta$ runs over the set of all primitive $n$-root of unity. Therefore 
\begin{align*} 
x^{\varphi(n)} \Phi_n \left(\frac{1}{x} \right) &= x^{\varphi(n)} \prod_{\zeta} (\frac{1}{x}-\zeta) = \prod_{\zeta} (1-x \zeta) 
= (-1)^{\varphi(n)} (\prod_{\zeta} \zeta) \prod_{\zeta} (x-\zeta^{-1}). 
\end{align*} 
Note that  $\zeta$ is a primitive $n$-root of unity if and only if $\zeta^{-1}$ is. Furthermore, as long as $n>2$, $\zeta \neq \zeta^{-1}.$ Therefore, we can partition the set of all primitive $n$-root of unity into pairs $\{\zeta, \zeta^{-1} \}.$ This shows that $\varphi(n)$ is even, $ \prod_{\zeta} \zeta =1$, and 
$\prod_{\zeta}(x-\zeta)= \prod_{\zeta} (x- \zeta^{-1}).$ 
We then conclude that 
$x^{\varphi(n)} \Phi_n \left( \frac{1}{x} \right) = \Phi_n(x).$ 
In other words, $\Phi_n(x)$ is a reciprocal polynomial. 
\end{proof}

Next, we discuss some results about $\Q$-bases of the cyclotomic extension $\Q(\zeta_n)/\Q.$ We first have the following classical theorem.
\begin{lem} \label{lem:cyclotomic_basis_square_free}
Let $n$ be a squarefree positive integer. Then the set of all primitive $n$-roots of unity, namely $\{ \zeta_n^k | \gcd(k,n)=1 \}$, is $\Q$-linear independent. 
\end{lem}
A proof of this lemma can be found in \cite{[Independent]}. Interested readers are encouraged to look at the thread \cite{[Conrad]} on mathoverflow for another proof provided by Keith Conrad as well as some interesting historical discussions. We will use Conrad's argument to prove the following lemmas.

\begin{lem}
Let $p$ be an odd prime. Then the set $\{\zeta_{4p}^k\mid  1\leq k\leq 2p,(k,2p)=1\}$ is $\Q$-linearly independent.
\end{lem}
\begin{proof}
We have $\phi(4p)=2(p-1)$ and the set $S=\{\zeta_{4p}^k\mid  1\leq k\leq 2p-2\}$ is a $\Q$-basis of $\Q(\zeta_{4p})$. Since $\Phi_{4p}(x)=\sum_{j=0}^{p-1}(-1)^jx^{2j}$, one has
\[
\sum_{j=0}^{p-1}(-1)^j\zeta_{4p}^{2j+1}=\zeta_{4p}\Phi_{4p}(\zeta_{4p})=0.
\]
Hence $\zeta_{4p}^{p}$ is a linear combination of $\{\zeta_{4p}^{2j+1}\mid j=0,\ldots,p-1 \text{ and } 2j+1\not=p\}$. So if we replace $\zeta_{4p}^{p}$ in the set $S$ by $\zeta_{4p}^{2p-1}$, we still have a $\Q$-basis of $\Q(\zeta_{4p})$.  In other words,  $\{\zeta_{4p}^k\mid  1\leq k\leq 2p-1, k\not= p\}$ is a $\Q$-basis of $\Q(\zeta_{4p})$. In particular, the set $\{\zeta_{4p}^k\mid 1\leq k\leq 2p,(k,2p)=1, 1\leq j\leq p, (j,p)=1\}$ is  $\Q$-linearly independent. 
\end{proof}

\begin{lem} \label{lem:2.7}
Let $m$ be an odd squarefree positive integer. Then the set $\{\zeta_{4m}^k\mid  1\leq k\leq 2m, (k,2m)=1\}$ is $\Q$-linearly independent.
\end{lem}
\begin{proof}
We proceed by induction on the number of prime factors of $m$. If $m$ is prime, we are done by the previous lemma. Suppose that the statement is true for an odd squarefree integer $m$. We show that the statement is true for $mp$, where $p$ is an odd prime number coprime to $m$. We know that $ \{\zeta_{4m}^k\mid (k,2m)=1, 1\leq k\leq 2m\}$ is $\Q$-linearly independent, $\{\zeta^j_p\mid  1\leq j\leq p, (j,p)=1\}$ is a $\Q$-basis of $\Q(\zeta_p)$. Since $\Q(\zeta_{4m})\cap \Q(\zeta_p)=\Q$, the set  \[S=\{\zeta_{4m}^k\zeta_p^j\mid  1\leq k\leq 2m,(k,2m)=1,1\leq j\leq p, (j,p)=1\}\] is $\Q$-linearly independent.
Inside $\Q(\zeta_{4mp})$ we can use $\zeta_{4m}=\zeta_{4mp}^p$ and $\zeta_{p}=\zeta_{4mp}^{4m}$. Then $\zeta_{4m}^k\zeta_p^j=\zeta_{4mp}^{kp+4mj}$. By taking modulo $4m$ and $p$, we see that  $kp+4mj$ is coprime to $4mp$. 
Let $l$ be the residue of $kp+4mj$ modulo $4mp$.  Then $l$ is coprime to $2mp$ and $\zeta_{4mp}^{kp+4mj}=\zeta_{4p}^l$. If $l>2mp$ then  $l'=l-2mp$ is coprime to $2mp$, $1\leq l'\leq 2mp$ and  $\zeta_{4mp}^{kp+4mj}=-\zeta_{4p}^{l'}$.

Note that if $kp+4mj\equiv k'p+4mj'\pmod {4mp}$, then $k\equiv k'\pmod {2m}$ and $j\equiv j\pmod p$. Hence the cardinality of $S$ is $\phi(2m)\phi(p)$, which is equality to the set 
$\{\zeta_{4mp}^l\mid  1\leq l\leq 2mp, (k,2mp)=1\}$. This implies that the latter set is $\Q$-linearly independent.
\end{proof}

\section{Generalized Fekete polynomials} 
\label{sec:fekete}
In this section, we introduce the generalized Fekete polynomials (see also \cite[Page 231]{[Lemmermeyer1]} and \cite[Page 24]{[BM]}.)
\begin{definition}
Let $\chi: (\Z/D)^{\times} \to \mathbb{C}^{\times}$ be a primitive Dirichlet character of conductor $D$. The generalized Fekete polynomial associated with $\chi$ is given by 
\[ F_{\chi}(x) =\sum_{a=1}^{D-1} \chi(a)x^a .\] 
\end{definition}
In this article, we will focus on the case of real primitive Dirichlet characters. 
\begin{definition} \label{def:fekete}
Let $d$ be a squarefree integer. Let $\Delta$ be the fundamental discriminant of the quadratic extension $\Q(\sqrt{d})/\Q$ and $\chi_{\Delta}$ the associated quadratic character. The generalized Fekete polynomial associated with $\Q(\sqrt{d})/\Q$ is given by 
\[ F_{\Delta}(x)= F_{\chi_{\Delta}}(x)= \sum_{a=1}^{D-1} \chi_{\Delta}(a) x^a =\sum_{a=1}^{D-1}\left(\frac{\Delta}{a} \right) x^a .\] 
\end{definition}
Note that in the above definition, we index $F$ by $\Delta$ instead of $D$ because $\Q(\sqrt{d})/\Q$ is determined by $\Delta$ and not necessarily by $D$. For example, $\Q(\sqrt{2})$ and $\Q(\sqrt{-2})$ both have $D=8.$ More generally, if $2m$ is a squarefree number, then  $\Q(\sqrt{2m})$ and $\Q(\sqrt{-2m})$ both have $D=|\Delta|=8m.$

We recall that the $L$-function associated with a primitive Dirichlet character $\chi$ with conductor $D$ is defined as 
\[ L(\chi,s)= \sum_{n=1}^{\infty} \frac{\chi(n)}{n^s}.\] 
There is a direct relationship between the generalized Fekete polynomial $F_{\chi}(x)$ and values of the $L$-function associated with $\chi$. More precisely 
\begin{prop}
\[ \Gamma(s) L(\chi, s)=  \int_{0}^1 \frac{(-\log(t))^{s-1} }{t} \frac{F_{\chi}(t)}{1-t^D} dt, \]
where $\Gamma(s)$ is the Gamma function. 
\end{prop}
\begin{proof}
For $0 <m \leq 1$, we recall that the Hurwitz zeta function $\zeta(s, m)$ is defined by 
\[ \zeta(s,m)=\sum_{n=0}^{\infty} \frac{1}{(n+m)^s} .\]
By \cite[Section 12.1]{[Apostol]}, we have 
\begin{equation}  \label{eq:decomposition}
L(s, \chi) = \frac{1}{D^s} \sum_{a=1}^{D-1} \chi(a) \zeta(s, \frac{a}{D}). 
\end{equation} 
Furthermore, by \cite[Theorem 12.2]{[Apostol]}, the Hurwitz zeta function $\zeta(s,m)$ has the following integral representation 
\[ \Gamma(s) \zeta(s,m) = \int_{0}^{\infty} \frac{x^{s-1} e^{-mx}}{1-e^{-x}} dx .\] 
Apply this formula for $m=\frac{a}{D}$ we have 
\[ \Gamma(s) \zeta(s,\frac{a}{D}) = \int_{0}^{\infty} \frac{x^{s-1} e^{-\frac{a}{D}x}}{1-e^{-x}} dx .\] 
Let $t=e^{-\frac{1}{D}x}.$ Then $x=-D \log(t)$ and $dx= \frac{-D}{t} dt.$ We then have 
\begin{equation} \label{eq:Hurwitz}
\Gamma(s) \zeta(s,\frac{a}{D}) = \int_{0}^{1} \frac{D^s (-\log(t))^{s-1} t^a}{t(1-t^D)} dt. 
\end{equation} 
Combining Equation \ref{eq:decomposition} and Equation \ref{eq:Hurwitz} we have 
\[ \Gamma(s) L(s, \chi) = \sum_{a=1}^{D-1} \chi(a) \int_{0}^{1} \frac{D^s (-\log(t))^{s-1} t^a}{t(1-t^D)} dt = \int_{0}^1 \frac{(-\log(t))^{s-1} }{t} \frac{F_{\chi}(t)}{1-t^D} dt
\qedhere.\] 

\end{proof}
\begin{rmk}
For some further discussions regarding real zeroes of $F_{\Delta}(x)$ over the interval $(0,1)$ and complex zeroes of $F_{\Delta}(x)$ on the unit circle, we refer the readers to \cite{[BM],[Conrey]}. 
\end{rmk}

Our numerical data suggests that some care must be taken when $\Delta$ is even. More precisely, when $\Delta$ is even, $\chi_{\Delta}(a)=0$ if $a$ is even. Consequently, $F_{\Delta}(x)/x$ is a polynomial in $x^2.$ In order to define and study the Fekete polynomials $f_{\Delta}(x)$ and the reduced Fekete polynomial $g_{\Delta}(x)$ associated with $\Delta$ (see \cite{[MTT3]} in the case $D=|\Delta|$ is a prime number), we introduce the following modification. 

\begin{definition} \label{def:modified_Fekete}
Suppose that $\Delta$ is an even number. The modified Fekete polynomial $\tilde{F}_{\Delta}(x)$ associated with $\Delta$ is given by 
\[ F_{\Delta}(x)= x \tilde{F}_{\Delta}(x^2) .\]  
Concretely 
\[ \tilde{F}_{\Delta}(x)= \sum_{a=0}^{D/2-1} \left(\frac{\Delta}{2a+1} \right) x^a .\] 
\end{definition}

\begin{rmk}
We motivate this definition by a concrete example; namely when $\Delta = 4 \times 11$. In this case, $F_{\Delta}(x)$ has the following factorization 
\begin{align*} F_{\Delta}(x)&= x(x - 1)^2 (x + 1)^2 \Phi_4(x) \Phi_{11}(x) \Phi_{22}(x) (x^{16} - x^{14} + 2x^{12} + 3x^8 + 2x^4 - x^2 + 1) \\
&= x \Phi_1(x^2)^2 \Phi_2(x^2) \Phi_{11}(x^2) f_\Delta(x^2), 
\end{align*} 
where $f_\Delta(x) = x^{8} - x^{7} + 2x^{6} + 3x^4 + 2x^2 - x + 1.$
By definition, we have 
\[ \tilde{F}_{\Delta}(x) = \Phi_1(x)^2 \Phi_2(x) \Phi_{11}(x) f_\Delta(x) .\] 
Philosophically speaking, we want to define the Fekete polynomial $f_{\Delta}(x)$ in such a way that $f_{\Delta}(x)$ is as simple as possible (and that $f_{\Delta}(x)$ contains all essential information about $F_{\Delta}(x)$ and hence $\tilde{F}_{\Delta}(x)$.)  This shows that, we should define 
\[ f_{\Delta}(x)= x^{8} - x^{7} + 2x^{6} + 3x^4 + 2x^2 - x + 1 .\]
Note that $f_{\Delta}(x)$ is a reciprocal polynomial and its trace polynomial is given by (see \cite[Section 3]{[MTT3]} for the definition) 
\[ g_{\Delta}(x)= x^4 - x^3 - 2x^2 + 3x + 1.\] 
The polynomials $f_{\Delta}(x)$ and $g_{\Delta}(x)$ will be the main subjects of investigation in our work. 
\end{rmk}

Finally, we also observe that $F_{\Delta}(x)$ satisfies the following symmetry. 

\begin{prop} \label{prop:reciprocal1} 
Let $\chi =\chi_{\Delta}$ be a quadratic Dirichlet character with conductor $D$. Then we have  
\[ x^{D} F_{\Delta} \left(\frac{1}{x} \right) = \chi(-1) F_{\Delta}(x) .\] 
If $D$ is even then 
\[ x^{\frac{D}{2}-1} \mFe \left(\frac{1}{x} \right) =  \chi(-1) \mFe(x) .\] 
\end{prop}

\begin{proof}
We have 
\begin{align*}
x^{D} F_{\chi} \left(\frac{1}{x} \right) &= x^D \sum_{a=1}^D \chi(a) \left(\frac{1}{x} \right)^a = \sum_{a=1}^{D-1} \chi(a) x^{D-a} \\ 
& = \sum_{u=1}^{D-1} \chi(D-u) x^u \quad  = \chi(-1) \sum_{u=1}^{D-1} \chi(u)x^u = \chi(-1) F_{\chi}(x).
\end{align*}
When $D$ is even, the second equality follows directly from the above equation and the definition $\mFe(x).$
\end{proof}

\subsection{Gauss sums and zeros of generalized Fekete polynomials} Let $\chi_{\Delta}$ be a quadratic Dirichlet character of conductor $D$ and $b$ an integer. Let $\zeta_{D}=\exp \left(\frac{2 \pi i}{D} \right)$ be a primitive $D$-root of unity. 
\begin{definition} (\cite[Chapter V]{[Ayoub]}) \label{def:gauss_sums}
The Gauss sum $G(b, \chi_{\Delta})$ is defined as follow 
\[ G(b, \chi_{\Delta})=\sum_{a=1}^{D-1} \chi_{\Delta}(a) \zeta_{D}^{ab}=F_{\Delta}(\zeta_{D}^b) .\] 
\end{definition}
We have the following fundamental property \cite[Theorem 4.12, page 312]{[Ayoub]}
\[
G(b,\chi_{\Delta})=\chi_{\Delta}(b) G(1,\chi_{\Delta}).
\]
A direct consequence of this property is that if $\gcd(b,D)>1$ then $F_{\Delta}(\zeta_{D}^b)=0$ (in fact, in the proof of \cite[Theorem 4.12, page 312]{[Ayoub]}, it was shown that if $\gcd(b,D)>1$ then  $G(b,\chi_{\Delta})=0$.) In other words, if $n|D$ and $n \neq D$ then $F_{\Delta}(\zeta_n)=0.$ On the other hand, we have \cite[Theorem 4.13, Page 313]{[Ayoub]}
\[ |F_{\Delta}(\zeta_{D})|^2= |G(1, \chi_{\Delta})|^2=D  .\] 
This shows that $\zeta_D$ is not a root of $F_{\Delta}(x).$ In light of these properties and the results  in \cite{[MTT3]}, the following question seems natural. 
\begin{question} \label{question:main}
Let $n$ be a positive integer. What is the multiplicity of $\zeta_{n}$ as a root of $F_{\Delta}(x)$? Equivalently, we can rephrase this question as a question about the multiplicity of $\Phi_n(x)$ in $F_{\Delta}(x)$. 
\end{question}
We remark that in the above question, we do not require $n$ to be a divisor of $D$. For simplicity, we will write $r_{\Delta}(\Phi_n)=r_{\Delta}(n)$ (respectively $\tilde{r}_{\Delta}(\Phi_n)=\tilde{r}_{\Delta}(n))$ for the multiplicity of $\Phi_n(x)$ in $F_{\Delta}(x)$ (respectively $\tilde{F}_{\Delta}(x)$.)  Quite surprisingly, the multiplicity is not always $1$. Furthermore, sometimes there are other exceptional zeros. For example, if $\Delta=3p$ with $p \equiv 2 \pmod{3}$ and $p \equiv 3 \pmod{4}$, then $\zeta_6$ is a simple root and $\zeta_3$ is a double root of $F_{\Delta}(x)$ (see Proposition \ref{prop:3p:zeta_3} for the precise statements and their proofs.)  

We also note that when $\Delta$ is even and $n$ is odd, $\Phi_n(x)$ and $\Phi_{2n}(x)=\Phi_n(-x)$ come in pair. This is a consequence of the following general fact. 
\begin{lem} \label{lem:pair}
Let $F \in \Z[x]$ be a polynomial in $x^2$; namely $F(x)= \tilde{F}(x^2)$ with $F(x) \in \Z[x]$. Suppose that $p(x)$ is an irreducible factor of $F(x)$ with multiplicity $m$. 
\begin{enumerate} 
\item If $p(x) \neq p(-x)$, then $p(-x)$ is also a factor of $F(x)$ with the same multiplicity. 
\item Assume that $p(x)=p(-x)$ and  $p(x)=\tilde{p}(x^2)$. Then the multiplicity of $p(x)$ in $F(x)$ is equal to the multiplicity of $\tilde{p}(x)$ in $\tilde{F}(x).$
\end{enumerate} 
\end{lem}
\begin{proof}
Because $p(x)$ is a factor of $F(x)$, we can write $F(x)= p(x) q(x).$  Take $x=-x$ and use the fact that $F(x)=F(-x)$, we have 
$F(x)=p(-x)q(-x).$ This shows that $p(-x)$ is a factor of $F(x).$ 

Let us first consider the case $p(x) \neq p(-x)$. In this case, $p(x)$ and $p(-x)$ are distinct factors of $F(x).$ Let us write 
$F(x) = p(x)^{m_1} p(-x)^{m_2} h(x),$ where $h(x) \in \Z[x]$ such that $\gcd(h(x), p(x)p(-x))=1.$ Take $x=-x$ again, we have 
$F(x) = F(-x) =p(-x)^{m_1} p(x)^{m_2} h(-x).$ 
Consequently $p(x)^{m_1} p(-x)^{m_2} h(x) = p(-x)^{m_1} p(x)^{m_2} h(-x).$  We conclude that $m_1=m_2.$

Let us now consider the second case, namely $p(x)= \tilde{p}(x^2).$ Let $m$ be the multiplicity of $p(x)$ in $F(x).$ We can write  $\tilde{F}(x^2)= F(x)=p(x)^m q(x) =\tilde{p}(x^2)^m q(x).$
Here $q(x) \in \Z[x]$ and $\gcd(q(x), p(x))=1$. From this equation, we see that $q(x)=\tilde{q}(x^2)$ for some $\tilde{q}(x) \in \Z[x]$; and hence $ \tilde{F}(x) = \tilde{p}(x)^m \tilde{q}(x)$ and $\gcd(\tilde{q}(x), \tilde{p}(x))=1.$ This shows that the multiplicity of $\tilde{p}(x)$ in $\tilde{F}(x)$ is $m$ as well. 
\end{proof}
By Proposition \ref{prop:cyclotomic}, we know that when $n$ is odd $\Phi_{2n}(x)=\Phi_n(-x)$ and $\Phi_n(x)\Phi_{2n}(x)=\Phi_n(x^2).$ Additionally, when $n=2m$ is even, we have  
$\Phi_{4m}(x)= \Phi_{2m}(x^2).$ As a corollary, we have the following
\begin{cor} \label{cor:pair}
Suppose that $\Delta$ is even and $n$ is odd. 
\begin{enumerate}
    \item If $\Phi_n(x)$ is a factor of $\tilde{F}_{\Delta}(x)$ then $\Phi_{2n}(x)$ is a factor of $F_{\Delta}(x)$. These factors have the same multiplicity.  
    \item The multiplicity of $\Phi_n(x)$ and $\Phi_{2n}(x)$ in $F_{\Delta}(x)$ is the same as the multiplicity of $\Phi_n(x)$ in $\tilde{F}_{\Delta}(x)$; namely 
    $ r_{\Delta}(\Phi_n)=r_{\Delta}(\Phi_{2n}) = \tilde{r}_{\Delta}(\Phi_n).$
\item   For any positive integer $m$,  
$ r_{\Delta}(\Phi_{4m})= \tilde{r}_{\Delta}(\Phi_{2m}).$ 
\end{enumerate}

\end{cor}

\begin{proof}
Let $m$ be the multiplicity of $\Phi_n(x)$ in $F_{\Delta}(x)/x = \tilde{F}_{\Delta}(x^2).$ Then by Lemma \ref{lem:pair}, $m$ is also the multiplicity of $\Phi_{2n}(x)$ in $F_{\Delta}(x)$. In other words, we can write 
\[ \tilde{F}_{\Delta}(x^2)= F_{\Delta}(x)/x = \Phi_n(x)^{m} \Phi_{2n}(x)^{m} q(x) = \Phi_n(x^2)^{m} q(x) .\] 
Here $q(x) \in \Z[x]$ is a polynomial which is relatively prime to $\Phi_n(x)\Phi_{2n}(x)=\Phi_n(x^2).$ We also note that the quotient of two polynomials in $x^2$ is another polynomial in $x^2$. Therefore, we can write $q(x)=\tilde{q}(x^2)$ where $\tilde{q}(x) \in \Z[x]$ and $\gcd(\tilde{q}(x), \Phi_n(x))=1.$ We then conclude that $\tilde{F}_{\Delta}(x) = \Phi_n(x)^{m} q_1(x).$ From this equation, we see that the multiplicity of $\Phi_n(x)$ in $\tilde{F}_{\Delta}(x)$ is $m$ as well. 
\end{proof}

\section{The multiplicity $\Phi_1(x)$ and $\Phi_2(x)$ in $F_{\Delta}(x)$ and $\tilde{F}_{\Delta}(x)$.} 
\label{sec:multiplicy_Phi1_Phi2}
We first investigate Question \ref{question:main} in the simplest cases; namely the multiplicity of $x=1$ and $x=-1$. First, we have 
$ F_{\Delta}(1)=\sum\limits_{a=1}^{D-1} \chi_{\Delta}(a)=0.$
Therefore, $x=1$ is a root of $F_{\Delta}(x)$. As before, let $r_{\Delta}(\Phi_1)=r_{\Delta}(1)$ be the multiplicity of the root $x=1$. 

To study $r_{\Delta}(\Phi_1)$, we need the following lemma.  
 
\begin{lem}{}
\label{lem:mult}
\begin{enumerate} 
\item[(a)] $\sum\limits_{a=1}^{D-1}\chi_{\Delta}(a)a =\begin{cases} 0 &\text { if $\chi_{\Delta}$ is even }\\
 \not=0 &\text { if $\chi_{\Delta}$ is odd}.
\end{cases}$
\item[(b)] $\sum\limits_{a=1}^{D-1}\chi_{\Delta}(a)a^2\not=0$.
\end{enumerate}
\end{lem}

We first recall that the Bernoulli polynomials $B_n(x)$ and Bernoulli numbers $B_n$, $n\geq 0$, are defined as  (see (\cite[pp. 7-9]{[Iwasawa]}))
\[
\dfrac{te^{xt}}{e^t-1}=\sum_{n=0}^\infty B_n(x) \dfrac{t^n}{n!}, \text{ and }
B_n(x) =\sum_{k=0}^n \binom{n}{k} B_k x^{n-k}, \quad n\geq 0.
\]
For example, $B_0(x)=1$, $B_1(x)=x+\dfrac{1}{2}$, $B_2(x)=x^2+x+\dfrac{1}{6}$.

 One can also define the generalized Bernoulli numbers $B_{n,\chi}$  and Bernoulli polynomials  $B_{n,\chi}(x)$, $n\geq0$, for a  Dirichlet character $\chi$ with conductor $f=f_\chi$ as (see \cite[pages 8-9]{[Iwasawa]})
 \[
 \sum_{a=1}^f \dfrac{\chi(a)te^{(a+x)t}}{e^{ft}-1}=\sum_{n=0}^\infty B_{n,\chi}(x) \dfrac{t^n}{n!},  \text{ and }
B_{n,\chi}(x) =\sum_{k=0}^n \binom{n}{k} B_{k,\chi} x^{n-k},\quad n\geq 0.
\]

\begin{proof}[Proof of Lemma~\ref{lem:mult}] 
Let $\chi$ be the quadratic character $\chi_{\Delta}$.  One has  (\cite[page 10]{[Iwasawa]})
\begin{equation}
\label{eq:1}
B_{n,\chi}(x)=D^{n-1}\sum_{a=1}^{D-1}\chi(a) B_n(\dfrac{a-D+x}{D}).
  \end{equation}
Substituting $n=1$ and $x=0$ in (\ref{eq:1}), one obtains
\[
\begin{aligned}
B_{1,\chi}&= \sum_{a=1}^{D-1}\chi(a) B_1 \left(\dfrac{a}{D}-1 \right)=\sum_{a=1}^{D-1}\chi(a)\left(\dfrac{a}{D}-\dfrac{1}{2}\right)\\
&=\dfrac{1}{D}\sum_{a=1}^{D-1}\chi(a)a -\dfrac{1}{2}\sum_{a=1}^{D-1}\chi(a)=\dfrac{1}{D}\sum_{a=1}^{D-1}\chi(a)a.
\end{aligned}
\]
Hence 
\begin{equation}
\label{eq:F'(1)}\sum_{a=1}^{D-1}\chi(a)a=DB_{1,\chi}=\begin{cases} 0 &\text { if $\chi$ is even}\\
 \not=0 &\text { if $\chi$ is odd }.
\end{cases}
\end{equation}
The last statement follows from \cite[pages 12-13, Theorem 2]{[Iwasawa]}. 

Substituting $n=2$ and $x=0$ in (\ref{eq:1}), one obtains
\[
\begin{aligned}
B_{2,\chi}&= D\sum_{a=1}^{D-1}\chi(a) B_2(\dfrac{a}{D}-1)
=D\sum_{a=1}^{D-1}\chi(a)\left(\dfrac{a^2}{D^2}-\dfrac{a}{D}+\dfrac{1}{6}\right)\\
&=\dfrac{1}{D}\sum_{a=1}^{D-1}\chi(a)a^2-\sum_{a=1}^{D-1}\chi(a)a  +\dfrac{1}{6}\sum_{a=1}^{D-1}\chi(a) =\dfrac{1}{D}\sum_{a=1}^{D-1}\chi(a)a^2 -\sum_{a=1}^{D-1}\chi(a)a.
\end{aligned}
\]
Hence
\[
\sum_{a=1}^{D-1}\chi(a)a^2 =D B_{2,\chi} + D\sum_{a=1}^{D-1}\chi(a)a.
\]
If $\chi$ is even, then $\sum\limits_{a=1}^{D-1}\chi(a)a^2 =D B_{2,\chi}\not=0$ (by \cite[pages 12-13, Theorem 2]{[Iwasawa]}).
If $\chi$ is odd, then $\sum\limits_{a=1}^{D-1}\chi(a)a^2=D\sum\limits_{a=1}^{D-1}\chi(a)a=D^2B_{1,\chi}\not=0$ (by \cite[pages 12-13, Theorem 2]{[Iwasawa]}).
\end{proof}
We are now ready to compute $r_{\Delta}(\Phi_1)$ explicitly. 

\begin{prop} 
\label{prop:mult_of_1}
$  r_{\Delta}(\Phi_1)= \begin{cases} 1 &\mbox{if $\chi_\Delta$ is odd} \\
2 & \mbox{if $\chi_\Delta$ is even}. 
\end{cases} $ 
\end{prop}
\begin{proof} For simplicity, we will write $\chi$ for $\chi_\Delta$.
One has 
$
F_{\Delta}^\prime(1)=\sum\limits_{a=1}^{D-1}\chi(a)a.
$
By Lemma~\ref{lem:mult} (a), $r_\Delta(\Phi_1)=1$ if $\chi$ is odd and $r_\Delta(\Phi_1)\geq 2$ if $\chi$ is even.

Now we suppose that $\chi$ is even. 
By Lemma~\ref{lem:mult}, one has
\[
\begin{aligned}
F_{\Delta}''(1)&=\sum_{a=1}^{D-1}\chi(a)a(a-1)
=\sum_{a=1}^{D-1}\chi(a) a^2 - \sum_{a=1}^{D-1}\chi(a) a
=\sum_{a=1}^{D-1}\chi(a) a^2 \not=0.
\end{aligned}
\]
Hence $r_{\Delta}(\Phi_1)=2$.
\end{proof}
By Corollary \ref{cor:pair}, we have the following. 
\begin{cor}
\label{cor:Even_Multiplicity_Phi1_Phi2}
Suppose that $\Delta$ is even. Let $\tilde{F}_{\Delta}(x)$ be the modified Fekete polynomial as defined in \cref{def:modified_Fekete}. Then 
$ r_{\Delta}(\Phi_2) = \tilde{r}_{\Delta}(\Phi_1)= r_{\Delta}(\Phi_1)= \begin{cases} 1 &\mbox{if $\chi_\Delta$ is odd} \\
2 & \mbox{if $\chi_\Delta$ is even}. \end{cases}$
\end{cor}

Next, let us investigate whether $x=-1$ is a root of $F_{\Delta}(x)$. We have the following observation. 

\begin{prop} \label{prop:roots_1}Suppose that $D$ is odd. 
\begin{enumerate}
\item If $\chi_\Delta$ is even then $x=-1$ is a root of $F_{\Delta}(x)$. 
\item If $\chi_\Delta$ is odd then $x=-1$ is not a root of $F_{\Delta}(x)$. 
\end{enumerate} 
\end{prop} 
\begin{proof} For simplicity, we will write $\chi$ for $\chi_\Delta$.
First, let us consider the case $\chi$ is even. We have 
\begin{align*}
F_{\Delta}(-1) &=\sum_{a=1}^{D-1} \chi(a)(-1)^a              =\sum_{a=1}^{\frac{D-1}{2}}  \left [\chi(a)(-1)^a+ \chi(D-a)(-1)^{D-a} \right] \\
              &=\sum_{a=1}^{\frac{D-1}{2}} \chi(a) \left[ (-1)^a+(-1)^{D-a} \right]=0.
\end{align*} 
 
Now, let us consider the case $\chi$ is odd. Note that in this case $\Delta<0$ and $\Delta\equiv 1\pmod 4$ and hence $D=-\Delta\equiv 3\pmod 4$. We have 
\begin{align*}
F_{\Delta}(-1) &=\sum_{a=1}^{D-1} \chi(a)(-1)^a
              =\sum_{a=1}^{\frac{D-1}{2}}  \left [\chi(2a)(-1)^{2a}+ \chi(D-2a)(-1)^{D-2a} \right] \\
              &=2\sum_{a=1}^{\frac{D-1}{2}} \chi(2a)=2 \chi(2) \sum_{a=1}^{\frac{D-1}{2}} \chi(a).
\end{align*} 
By \cite[Table 3, page 199]{[UW]}, we have  $\sum\limits_{a=1}^{\frac{D-1}{2}} \chi(a)=  \dfrac{1}{c}h(-D),$ where $h(-D)$ is the class number of the imaginary quadratic field $\Q(\sqrt{-D})$ and $c=\begin{cases} 1 & \text{ if } D\equiv 7\pmod 8\\
\frac{1}{3} &\text{ if } D\equiv 3 \pmod 8\end{cases}$. Therefore, 
\begin{equation}
\label{eq:F(-1)}
F_{\Delta}(-1)= 2 \chi(2) \dfrac{1}{c}h(-D). 
\end{equation}
 We conclude that $F_p(-1) \neq 0$ if $\chi$ is odd. 
\end{proof}

\begin{lem} \label{half_sum}
Suppose $\chi=\chi_\Delta$ is even. Then 
$\sum\limits_{a=1}^{\frac{D-1}{2}} \chi(a)=0.$  
\end{lem} 
\begin{proof} 
We have 
$
 0= \sum\limits_{a=1}^{D-1} \chi(a) =\sum\limits_{a=1}^{\frac{D-1}{2}}  \left[ \chi(a)+ \chi(D-a) \right]=2\sum\limits_{a=1}^{\frac{D-1}{2}} \chi(a). 
$
Hence $\sum\limits_{a=1}^{\frac{D-1}{2}} \chi(a)=0.$ 
\end{proof} 

\begin{prop} \label{prop: x=-1}
If $D$ is odd and $\chi_\Delta$ is even then $x=-1$ is a simple root of $F_{\Delta}(x)$.
\end{prop} 
\begin{proof}
First note that in the case that $\chi=\chi_\Delta$ is even, $D=\Delta\equiv 1\pmod 4$.
We already show that $F_{\Delta}(-1)=0$. To show that $x=-1$ is a simple root of $F_{\Delta}(x)$, we need to show $F_{\Delta}'(-1) \neq 0$. We have 
$ F_{\Delta}'(-1)=\sum\limits_{a=1}^{D-1} \chi(a) a (-1)^{a-1}=- \sum\limits_{a=1}^{D-1} (-1)^a \chi(a) a.$
Therefore, it is enough to show that $\sum\limits_{a=1}^{D-1} \chi(a) a (-1)^{a} \neq 0$.  We have 
\begin{align*}
\sum_{a=1}^{p-1} (-1)^a \chi(a)a 
              &=\sum_{a=1}^{\frac{D-1}{2}}  \left [ (-1)^{2a} \chi(2a)(2a)+ (-1)^{D-2a}\chi(D-2a)(D-2a) \right] \\
              &=\sum_{a=1}^{\frac{D-1}{2}} \chi(2a)(4a) - D \chi(2) \sum_{a=1}^{\frac{D-1}{2}} \chi(a)= 4 \chi(2) \sum_{a=1}^{\frac{D-1}{2}} \chi(a) a.
\end{align*} 
By \cite[page 28]{[UW]}, we have $\sum_{a=1}^{\frac{D-1}{2}} \chi(a) a = \dfrac{\bar{\chi}(2)-4}{4}B_{2,\chi}\not=0.$ Consequently,  
 $F_{\Delta}'(-1) \neq 0$. We conclude that $x=-1$ is a simple root of $F_{\Delta}(x)$. 
\end{proof} 

In summary, we have the following.
\begin{prop}
Suppose that $\Delta$ is odd. Then 
$  r_{\Delta}(\Phi_2) = \begin{cases} 0 &\mbox{if $\chi$ is odd} \\
1 & \mbox{if $\chi$ is even.} \end{cases} $  

\end{prop}

\section{Some general results about $r(\Phi_n)$ and $\tilde{r}(\Phi_n)$ for big $n$ and $n|D$} 
\label{sec:general_n}
In this section, we discuss some  general results about $r(\Phi_n)$ and $\tilde{r}(\Phi_n)$ for big $n$. We first have the following theorem.

\begin{thm}
\label{prop:positive_big_divisor}
Suppose that $\Delta=dn$ is odd and $n>d>1.$ Then $\zeta_n$ is a simple root of $F_{\Delta}(x).$ In other words, $r_{\Delta}(\Phi_n) = 1.$ 
\end{thm} 

\begin{proof}
Recall that 
$  F_{\Delta}(x)= \sum\limits_{a=1}^{\Delta-1} \left(\frac{\Delta}{a} \right)x^a.$ 
We already know that $F_{\Delta}(\zeta_n)=0.$ We need to show that $F'_{\Delta}(\zeta_n) \neq 0$. We have 
$ F'_{\Delta}(x)= \sum\limits_{a=1}^{\Delta-1} a \left(\frac{\Delta}{a} \right)x^{a-1}.$  
We then see that 
\begin{align*}
\zeta_n F'_{\Delta}(\zeta_n)&= \sum_{a=1}^{\Delta-1}\left(\frac{\Delta}{a} \right) a \zeta_n^a  
=\sum_{m=1}^{n} \left[\sum_{a \equiv m \pmod{n}} \left(\frac{\Delta}{a} \right) a \right] \zeta_n^k 
= \sum_{m=1}^{n} A_m \zeta_n^k,
\end{align*} 
where 
\[ A_m= \sum_{a \equiv m \pmod{n}} \left(\frac{\Delta}{a} \right) a , \quad 1 \leq m \leq n. \] 
We know that $A_m=0$ for every $m$ with $(m,n)>1$.  Now suppose that $F'_\Delta(\zeta_n)=0$. Because $n$ is squarefree, the set $\{\zeta_n^m\mid (m,n)=1,1\leq m\leq n\}$ is a $\Q$-linearly independent (see  \cref{lem:cyclotomic_basis_square_free}). Hence $A_m=0$ for every $m$ with $(m,n)=1$.

We observe that we can rewrite $A_m$ as follows 
\begin{align*}
 A_m &=\sum_{k=0}^{d-1} (kn+m) \left(\frac{\Delta}{kn+m} \right) 
     = \sum_{k=0}^{d-1} (kn+m) \left(\frac{kn+m}{dn} \right) \\
     &= \left(\frac{m}{n} \right) \sum_{k=0}^{d-1} (kn+m) \left(\frac{kn+m}{d} \right) = m \left(\frac{m}{n} \right) \sum_{k=0}^{d-1}  \left(\frac{kn+m}{d} \right)+ n \left(\frac{m}{n} \right) \sum_{k=0}^{d-1} \left(\frac{kn+m}{d} \right) k.
\end{align*} 

We note that $\{ kn+m \}_{k=0}^{d-1}$ is a complete system of residues modulo $d$. Therefore 
\[ \sum_{k=0}^{d-1}  \left(\frac{kn+m}{d} \right)=0 .\] 
Consequently, we have 

\[ E_m= \sum_{k=0}^{d-1} \left(\frac{kn+m}{d} \right) k = \sum_{k=1}^{d-1} \left(\frac{kn+m}{d} \right) k= 0 , \forall 1 \leq m \leq n-1. \] 
Let $\widetilde{m}$ be the unique number in $\{0, 1, \ldots, d-1\}$ such that $\widetilde{m} n \equiv m \pmod{d}$ (in other words $\widetilde{m}=n^{-1}m$ modulo $d$). We then have 
\[ E_m = \left(\frac{n}{m} \right) \sum_{k=0}^{d-1} \left(\frac{k+\widetilde{m}}{d} \right) k =\left(\frac{n}{m} \right) \sum_{k=0}^{d-1} \left(\frac{k+\widetilde{m}}{d} \right) (k+ \widetilde{m}) .\] 
The last inequality follows from the fact that  $\sum_{k=0}^{d-1} \left(\frac{k+\widetilde{m}}{d} \right) =0.$ Let 
\[ C_{\widetilde{m}} = \sum_{k=0}^{d-1} \left(\frac{k+\widetilde{m}}{d} \right) (k+ \widetilde{m}) .\] 
Then $C_{\widetilde{m}}=0$ by our assumption (note that $\widetilde{m}$ depends on $m$.) Under the assumption that $n>d$, 
let us consider $m= n-d.$ Then $\tilde{m}=1.$ We have 
\begin{align*} C_1 &= \sum_{k=0}^{d-1}  \left(\frac{k+1}{d} \right) (k+1)  
= \sum_{u=1}^{d-1} \left(\frac{u}{d} \right)u = d B_{1, \chi_{d^{*}}}.
\end{align*}
Here $d^{*}= (-1)^{\frac{d-1}{2}}d.$ If $d^{*}<0$ then $C_1 \neq 0.$ In this case, we are done. Now suppose $d^{*}>0.$

We have the following small lemma. 
\begin{lem}
There exists $m$ such that  $1 \leq m <n$, $\gcd(m,n)=1$ and $  2n \equiv m \pmod{d}.$ 
For this $m$, $\tilde{m}=2.$
\end{lem}
\begin{proof}
Let us write $n=qd+r$, with $1\leq r\leq d-1$. Under the condition, $n>d$, we have $q \geq 1.$ We will look for of the form $m= hd+2r$ for some $0 \leq h \leq q-1$. 
 Note that for this choice of $m$, conditions $(1)$ and $(3)$ are satisfied. We will look for an $h$ such that $m$ satisfies condition (2) as well. 
 We have 
\begin{align*} 
\gcd(n, m)&= \gcd(n, hd+2r)=\gcd(n, 2n-(hd+2r)) \\
&= \gcd(n, (2q-h)d)=\gcd(n, 2q-h). 
\end{align*}
Note that $q+1 \leq 2q-h \leq 2q.$
We claim that  there exists an integer $t$ such that $q+1 \leq 2^t \leq 2q.$
In fact, let $s$ be the largest integer such that $2^s \leq q+1.$ If $q+1=2^s$, we can take $t=s$. Otherwise, $2^s<q+1.$ Hence 
$ 2^{s+1} <2(q+1) =2q+2.$ From this, we can easily see that $2^{s+1} \leq 2q.$  In this case, we can take $t=s+1.$

Now, let $0 \leq h \leq q-1$ be the number such that $2q-h =2^t.$
Let $m=hd+2r.$ Our argument shows that $\gcd(n,m)=\gcd(n, 2^t)=1$.
\end{proof}

For this choice of $m$, $\tilde{m}=2$. Let us consider 

\begin{align*} 
C_2 &= \sum_{k=0}^{d-1}  \left(\frac{k+2}{d} \right) (k+2)  
= \sum_{u=2}^{d+1} \left(\frac{u}{d} \right)u \\
&= \sum_{u=1}^d \left(\frac{u}{d} \right)u - 1 + \left(\frac{d+1}{d} \right)(d+1)
= d B_{1, \chi_{d^{*}}}+d=d \neq 0.
\end{align*}
This completes the proof. 
\end{proof}

 When $\Delta <0$, the situation becomes a bit easier. 

\begin{thm}
\label{prop:negative_big_divisor}
If $\Delta<0$ and $n$ is a squarefree proper divisor of $\Delta$ then $\zeta_n$ is a simple root of $F_\Delta$.
\end{thm}
\begin{proof}
We already know that $F_{\Delta}(\zeta_n)=0.$ We need to show that $F'_{\Delta}(\zeta_n) \neq 0$. We have 
$ F'_{\Delta}(x)= \sum\limits_{a=1}^{\Delta-1} a \left(\dfrac{\Delta}{a} \right)x^{a-1}.$ 
We then see that 
\begin{align*}
\zeta_n F'_{\Delta}(\zeta_n)&= \sum_{a=1}^{\Delta-1}\left(\frac{\Delta}{a} \right) a \zeta_n^a  
=\sum_{k=1}^{n} \left[\sum_{a \equiv k \pmod{n}} \left(\frac{\Delta}{a} \right) a \right] \zeta_n^k 
= \sum_{k=1}^{n} A_k \zeta_n^k,
\end{align*} 
where 
\[ A_k= \sum_{a \equiv k \pmod{n}} \left(\frac{\Delta}{a} \right) a , \quad 1 \leq k \leq n. \] 
We know that $A_k=0$ for every $k$ with $(k,n)>1$.  Now suppose that $F'_\Delta(\zeta_n)=0$. Because $n$ is squarefree, the set $\{\zeta_n^k\mid (k,n)=1,1\leq k\leq n\}$ is a $\Q$-basis of $\Q(\zeta_n)$. Hence $A_k=0$ for every $k$ with $(k,n)=1$. This implies that 
$ F'_{\Delta}(1)= \sum\limits_{m=1}^{n} A_m = 0. $
This contradicts to 
 \cref{prop:mult_of_1}. Therefore $F'_\Delta(\zeta_n)$ must be nonzero, as required. 
\end{proof}
When $\Delta$ is negative and even, we have the following partial result. 
\begin{thm}
If $\Delta=-4d<0$, where $d$ is odd squarefree and $d\equiv 1 \pmod 4$. Assume that $d=mq$ where $q$ is a prime number and $m$ is a positive number. Let $S_{m}= \{1 \leq k \leq 2m| \gcd(2m, k)=1 \}.$  Let us assume further that $S_{m}$ contains at least $q/2$ distinct residues modulo $q.$ Then $\zeta_{4m}$ is a simple root of $F_{\Delta}(x).$
\end{thm}
\begin{proof}
Let $n = 4m$. We already know that $F_{\Delta}(\zeta_n)=0.$ We need to show that $F'_{\Delta}(\zeta_n) \neq 0$. We have 
\begin{align*}
\zeta_n F'_{\Delta}(\zeta_n)&= \sum_{a=1}^{\Delta-1}\left(\frac{\Delta}{a} \right) a \zeta_n^a 
=\sum_{k=1}^{n} \left[\sum_{a \equiv k \pmod{n}} \left(\frac{\Delta}{a} \right) a \right] \zeta_n^k 
= \sum_{k=1}^{n} A_k \zeta_n^k,
\end{align*} 
where  
\[ A_k= \sum_{a \equiv k \pmod{n}} \left(\frac{\Delta}{a} \right) a , \quad 1 \leq k \leq n. \] 
Note that $\zeta_{n}^k=-\zeta_{n}^{k+2m}$. Hence $\zeta_n F'_{\Delta}(\zeta_n)=\sum_{k=1}^{2m} (A_k-A_{k+2m}) \zeta_n^k.$ We know that $A_k=0$ for every $k$ with $(k,n)>1$.   Now suppose that $F'_\Delta(\zeta_n)=0$. By  \cref{lem:2.7}, the set $\{\zeta_n^k\mid (k,2m)=1,1\leq k\leq 2m\}$ is  $\Q$-linearly independent in $\Q(\zeta_n)$. Hence $A_k=A_{k+2m}$ for every $k$ with $(k,2m)=1$ and $1\leq k\leq 2m$. 

From our assumption, we have $d=mq$. We then have 

\begin{align*}
A_{k}-A_{k+2m} &= \sum_{a \equiv k \pmod{4m}} \left[\left(\frac{-4mq}{a} \right)a - \left(\frac{-4mq}{a+2m} \right)(a+2m) \right]  \\
&= \sum_{a \equiv k \pmod{4m}} \left[\left(\frac{-4mq}{a} \right)- \left(\frac{-4mq}{a+2m} \right) \right] a -2m \sum_{a \equiv k \pmod{4m}} \left(\frac{-4mq}{a+2m} \right).
\end{align*}
Let us write $a= k+ 4mt$ for $0 \leq t \leq q-1$. We have 
\begin{align*} 
\left(\frac{-4mq}{a} \right) &=\left(\frac{-4mq}{k+4mt} \right)=  \left(\frac{-qm}{k+4mt} \right) = (-1)^{\frac{k-1}{2}} \left(\frac{k+4mt}{qm} \right) \\
&= (-1)^{\frac{k-1}{2}} \left(\frac{k+4mt}{q} \right) \left(\frac{k+4mt}{m} \right)= (-1)^{\frac{k-1}{2}} \left(\frac{k+4mt}{q} \right) \left(\frac{k}{m} \right).
\end{align*} 

Similarly 
\begin{align*}
\left(\frac{-4mq}{a+2m} \right)&= \left(\frac{-4mq}{k+2m+4mt} \right)=(-1)^{\frac{k-1}{2}+m} \left(\frac{k+2m+4mt}{q} \right) \left(\frac{k+2m}{m} \right) \\
&=-(-1)^{\frac{k-1}{2}} \left(\frac{k}{m} \right) \left(\frac{k+2m+4mt}{q} \right).
\end{align*}
Because $\{k+2m+4mt \}_{t=0}^{q-1}$ is a complete residue system modulo $q$, we have 

\[ \sum_{a \equiv k \pmod{4m}} \left(\frac{-4mq}{a+2m} \right) =-(-1)^{\frac{k-1}{2}} \left(\frac{k}{m} \right)\sum_{0 \leq t \leq q-1} \left(\frac{k+2m+4mt}{q} \right) =0 .\]

Therefore, we conclude that 
\begin{align*}
A_{k}-A_{k+2m} &= (-1)^{\frac{k-1}{2}} \left(\frac{k}{m} \right) \sum_{0 \leq t \leq q-1} \left[ \left(\frac{k+4mt}{q} \right)+ \left(\frac{k+2m+4mt}{q} \right) \right] (k+4mt) \\
&= (-1)^{\frac{k-1}{2}} \left(\frac{k}{m} \right) 4m \sum_{0 \leq t \leq q-1} \left[ \left(\frac{k+4mt}{q} \right)+ \left(\frac{k+2m+4mt}{q} \right) \right]t.
\end{align*}
The second equality follows from the fact that 
\[ \sum_{0 \leq t \leq q-1} \left(\frac{k+4mt}{q} \right) = \sum_{0 \leq t \leq q-1} \left(\frac{k+2m+4mt}{q} \right) =0 .\] 

By our assumption, $A_{k}-A_{k+2m}=0$. Hence 
\begin{align*}
\sum_{0 \leq t \leq q-1} \left[ \left(\frac{k+4mt}{q} \right)+ \left(\frac{k+2m+4mt}{q} \right) \right]t =0.
\end{align*}
Note that all terms in the above sum are even, except the case where either $\left(\frac{k+4mt}{q} \right)$ or $\left(\frac{k+2m+4mt}{q} \right)$ is zero. Let $t_{1,k}^{*}$ (respectively $t_{2,k}^{*}$) be the unique integer in $\{0, 1, 2, \ldots, q-1 \}$ such that $q | 4mt_{1,k}^{*} +k$ (respectively $q| 4mt_{2,k}^{*}+2m+k$). Then 

\begin{align*}
0= \sum_{0 \leq t \leq q-1} \left[ \left(\frac{k+4mt}{q} \right)+ \left(\frac{k+2m+4mt}{q} \right) \right]t & \equiv t_{1,k}^{*}-t_{2,k}^{*} \pmod{2}.
\end{align*}
We remark that by our definition $q| 4m(t_{1,k}^{*} -t_{2,k}^{*})-2m .$ Equivalently $q| 2(t_{1,k}^*-t_{2,k}^{*})-1.$  We conclude that either $t_{1,k}^{*}-t_{2,k}^{*} = \frac{q+1}{2}$ or $t_{1,k}^{*}-t_{2,k}^{*}= \frac{1-q}{2}.$ Furthermore, the first case happens if $t_{1,k}^{*}>q/2$ and the second case happens otherwise. Under the assumption that $S_{m}$ contains at least $q/2$ distinct residues modulo $q$, we know that there exist $k_1, k_2 \in S_m$ such that $t_{1,k_1}^{*}>q/2$ and $t_{1,k_2}^{*}<q/2.$ In the first case 
\[ 0 \equiv t_{1,k_1}^{*}-t_{2, k_1}^{*} \equiv \frac{q+1}{2} \pmod{2} .\] 
In the second case 
\[ 0 \equiv t_{1,k_2}^{*}-t_{2, k_2}^{*} \equiv \frac{1-q}{2} =1+\frac{q+1}{2} \pmod{2} .\] 
This is a contradiction.  
\end{proof}
We provide below a concrete example where the condition on $S_m$ is satisfied. 
\begin{lem}
Let $m=p_1 \ldots p_r$. Assume that $\min \{p_1, p_2, \ldots, p_r \}>q$. Then $S_m$ contains at least $q/2$ residues modulo $q.$
\end{lem}
\begin{proof}
Let us consider the following set  
$ S = \{1, 3, \ldots, q \} = \{2k+1| 1 \leq 2k+1 \leq q \}.$
We can see that $|S|=\frac{q+1}{2}>q/2$. Furthermore, by our assumption $S \subset S_m$. Consequently, $|S_m|>q/2.$
\end{proof}

\section{Partial results on $r(\Phi_n)$ when $n \nmid \Delta.$} 
\label{sec:n_not_divisor}
\begin{lem}
Let $R$ be a (unital) ring and $I$ is an ideal of $R$. If $a^m-1$ and $a^n-1$ are in $I$, where $m$ and $n$ are positive integers, then $a^d-1$ is also in $I$, where $d=\gcd(m,n)$.
\end{lem}
\begin{proof}
We may suppose $m\geq n$. Write $m=nq+r$, where $0\leq r<b$. We have $a^m-1=a^r(a^{nq}-1)+a^r-1$. This implies that $a^r-1$ is also in $I$. By repeating this argument, we conclude that $a^d-1$ is in $I$.
\end{proof}
  
\begin{prop} Let $D=|\Delta|=p$ be a prime number and $m$ an positive integer. If  $\gcd(m,p-1)<\phi(m)$, then  $F_\Delta(\zeta_m)\not=0$.
\end{prop}

\begin{proof} One has
$F_\Delta (x) \equiv x+x^2+\cdots+x^{p-1}=x\frac{x^{p-1}-1}{x-1}\pmod 2.
$
Suppose that $F_\Delta(\zeta_m)=0$. Then one has 
$
\zeta_m^{p-1}-1 \in I,
$
where $I=2\Z[\zeta_m]$. 
Clearly, $\zeta_m^m-1=0\in I$. By previous lemma, $\zeta_m^d-1\in I$, where $d=\gcd(m,p-1)<\phi(m)$. Therefore
\[
\zeta_m^d=1+ 2(a_0+a_1\zeta_m+\cdots+a_{\varphi(m)}\zeta_m^{\varphi(m)-1}).
\]
Since $\{\zeta_m^k\mid k=0,1,\ldots,\varphi(m)-1\}$ is a $\Q$-basis of $\Q(\zeta_m)$, this implies that $1=2a_d$, a contradiction. 
\end{proof}

\begin{prop} Let $D=|\Delta|=p$ be a prime number and $m$ an positive integer. If  $\dfrac{m}{\gcd(m,p-1)}$ is an odd integer greater than 1, then  $F_\Delta(\zeta_m)\not=0$. In particular, if $m$ is odd and $m$ does not divide $p-1$ then $F_\Delta(\zeta_m)\not=0$
\end{prop}
\begin{proof} Let $d=\gcd(m,p-1)$ and let $r=m/d$. Suppose that $F_\Delta(\zeta_m)=0$.
As in the previous proof, $\zeta_m^d-1\in I:=2\Z[\zeta_m]$. Write $\zeta_m^d=1+2a$, for some $a\in \Z[\zeta_m]$. Note that $a\not=0
$ because $r>1$. 
One has 
$1=(\zeta_m^d)^r=1+\sum\limits_{k=1}^r\binom{r}{k}(2a)^k.$ 
Hence
$ r=\sum\limits_{k=2}^{r-1}\binom{r}{k}(2a)^{k-1}\in I.$
Since $r$ is odd, this implies that $1$ is in $I$, a contradiction.
\end{proof}

\section{Generalized Fekete polynomials with $\Delta=4p$, $p \equiv 3 \pmod{4}$} 
\label{sec:4p}

Before we discuss the proofs for our results in this and the following sections, we would like to explain our general strategy for computation of $r_{\Delta}(\Phi_n)$ in the case $n | \Delta.$ First, we know that by the theory of Gauss sums (see the discussion after \cref{def:gauss_sums}), if $n|\Delta$ and $ n \neq |\Delta|$ then $r_{\Delta}(\Phi_n) \geq 1$. By definition, $r_{\Delta}(\Phi_n)$ is equal to the positive integer $k$ such that $ F_{\Delta}(\zeta_n) = F'_{\Delta}(\zeta_n) = \ldots = F_{\Delta}^{(k-1)}(\zeta_n) =0$ and $F_{\Delta}^{(k)}(\zeta_n) \neq 0.$ It turns out that the values $F_{\Delta}^{i}(\zeta_n)$ for $1 \leq i \leq k$ can often be expressed in terms of Bernoulli numbers and hence special values of $L$-functions. Using this relationship, we obtain information about the vanishing and non-vanishing of $F_{\Delta}^{(i)}(\zeta_n)$. Schematically, our strategy can be summarized by the following diagram 

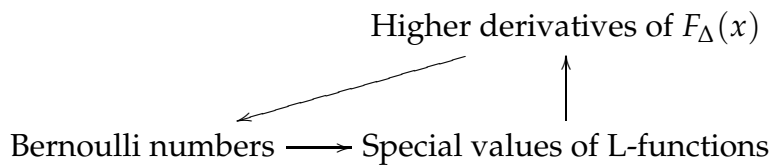
\[ \xymatrix
   { & \text{Higher derivatives of $F_{\Delta}(x)$} \ar[dl]  & \\
     \text{Bernoulli numbers}  \ar[r] &  \text{Special values of L-functions} \ar[u] 
   }
\] 

This strategy will be applied repeatedly in Sections $7-10.$  We will explain in detail our prototype calculations in Section 7. For the other sections, we will keep the essential parts of our proofs and omit minor details in order to avoid repetition.

\subsection{Cyclotomic factors of $\Fe(x)$}

By Corollary~\ref{cor:Even_Multiplicity_Phi1_Phi2} in Section~\ref{sec:multiplicy_Phi1_Phi2}, we know that 
\[ r_{\Delta}(\Phi_1)= r_{\Delta}(\Phi_2)= \tilde{r}_{\Delta}(\Phi_1) = 2 .\] 
We next determine the multiplicity of $\Phi_p(x), \Phi_{2p}(x)$ in $F_{\Delta}(x).$ We first recall the definition of $\mFe(x)$
\[ \mFe(x) = \sum_{a=0}^{2p-1} \left(\frac{4p}{2a+1} \right) x^a .\] 
By the reciprocity law, we have 
\[ \left(\frac{4p}{2a+1} \right)= \left(\frac{p}{2a+1} \right)= \left(\frac{2a+1}{p} \right) (-1)^{\frac{a(p-1)}{2}}= (-1)^{a} \left(\frac{2a+1}{p} \right)  .\]
Note also that 
$ \left(\dfrac{2a+1}{p} \right)= \left(\dfrac{2}{p} \right) \left(\dfrac{a+\frac{p+1}{2}}{p} \right).$ 
Combining these facts, we have 
\begin{align*}
\mFe(x) &= \left(\frac{2}{p} \right) \sum_{a=0}^{2p-1} \left(\frac{a+\frac{p+1}{2}}{p} \right)(-1)^a x^a 
        = \left(\frac{2}{p} \right) \sum_{a=0}^{p-1} \left(\frac{a+\frac{p+1}{2}}{p} \right)[(-1)^a x^a -(-1)^a x^{a+p}].
\end{align*}

\begin{prop}
$   r_{\Delta}(\Phi_p)=r_{\Delta}(\Phi_{2p})= \tilde{r}_{\Delta}(\Phi_p) =1.$ 
\end{prop}

\begin{proof}
We have 
$ 
\mFe'(x)=\left(\dfrac{2}{p} \right) \sum\limits_{a=1}^{p-1} \left(\frac{a+\frac{p+1}{2}}{p} \right)[(-1)^a a x^{a-1} -(-1)^a (a+p) x^{a+p-1}].
$ 
Hence
\[
\begin{aligned}
\mFe'(\zeta_p)&=\left(\frac{2}{p} \right) \sum_{a=1}^{p-1} \left(\frac{a+\frac{p+1}{2}}{p} \right)[(-1)^a a \zeta_p^{a-1} -(-1)^a (a+p) \zeta_p^{a+p-1}]\\
&=-\left(\frac{2}{p} \right)p \sum_{a=1}^{p-1} \left(\frac{a+\frac{p+1}{2}}{p} \right)(-1)^a  \zeta_p^{a-1}.
\end{aligned}
\]
By Lemma \ref{lem:cyclotomic_basis_square_free} we know that $\{\zeta_p, \zeta_p^2, \ldots, \zeta_p^{p-1} \}$ is $\Q$-linearly independent. Therefore, we can conclude that $\mFe'(\zeta_p) \neq 0$ and the above statement holds. 
\end{proof}

For the multiplicity of $\Phi_4(x)$, we have the following.

\begin{prop}
$  r_{\Delta}(\Phi_4)= \tilde{r}_{\Delta}(\Phi_2) = \begin{cases} 3 &\mbox{if } p \equiv 7 \pmod{8} \\
1 & \mbox{if } p \equiv 3 \pmod{8}. \end{cases} $ 
\end{prop}
\begin{proof}
We have 
$ \mFe(-1) = 2 \left(\dfrac{2}{p} \right) \sum\limits_{a=0}^{p-1} \left(\dfrac{a+\frac{p+1}{2}}{p} \right)=0.$
We also have 
{\allowdisplaybreaks
\begin{align*}
\left(\frac{2}{p} \right) \mFe'(-1) &= \sum_{a=0}^{p-1}     \left(\frac{a+\frac{p+1}{2}}{p} \right) (-a -(a+p)) 
= -2 \sum_{a=0}^{p-1}\left(\frac{a+\frac{p+1}{2}}{p} \right)a -p \sum_{a=0}^{p-1} \left(\frac{a+\frac{p+1}{2}}{p} \right)\\ 
&= -2 \sum_{a=0}^{p-1}\left(\frac{a+\frac{p+1}{2}}{p} \right)a 
= - 2 \sum_{a=0}^{\frac{p-1}{2}} \left(\frac{a+\frac{p+1}{2}}{p} \right)a - 2 \sum_{a=\frac{p+1}{2}}^{p-1} \left(\frac{a+\frac{p+1}{2}}{p} \right)a \\
&=- 2 \sum_{u=\frac{p+1}{2}}^{p} \left(\frac{u}{p} \right)(u-\frac{p+1}{2}) - 2 \sum_{u=1}^{\frac{p-1}{2}} \left(\frac{u}{p} \right)(u+\frac{p-1}{2}) \\
&= - 2 \sum_{u=1}^{p-1} \left(\frac{u}{p} \right)u + (p+1) \sum_{u=\frac{p+1}{2}}^p \left(\frac{u}{p} \right) -(p-1) \sum_{u=1}^{\frac{p-1}{2}} \left(\frac{u}{p} \right)\\
&= - 2 \sum_{u=1}^{p-1} \left(\frac{u}{p} \right)u  -2p \sum_{u=1}^{\frac{p-1}{2}} \left(\frac{u}{p} \right)
 =2ph(-p)-2p\left(2-\left(\dfrac{2}{p}\right)\right)h(-p)\\
 &=\begin{cases} 0 &\text{ if } p\equiv 7\pmod 8\\
 -4ph(-p)  &\text{ if } p\equiv 3\pmod 8.
 \end{cases}
 \end{align*}
 }
For the second to the last equality we use the fact that for $p \equiv 3 \pmod{4}$ ), one has
\vspace{-.1cm}
\begingroup
\addtolength{\jot}{0em}
\begin{align*}
\sum_{u=1}^{p-1} \left(\frac{u}{p}u \right)&=-ph(-p) & (\text{see \cite[Equation 3]{[Girstmair]}})\\
\sum_{u=1}^{\frac{p-1}{2}} \left(\frac{u}{p} \right)&= \left(2-\left(\dfrac{2}{p}\right)\right)h(-p)&(\text{see \cite[Corollary 3.4]{[Berndt]}}).
\end{align*} 
\endgroup
Now we suppose that $p\equiv 7\pmod 8$. We have 
{\allowdisplaybreaks
\begin{align*}
\left(\frac{2}{p} \right) \mFe''(-1) &= 2 \sum_{a=0}^{p-1} \left(\frac{a+\frac{p+1}{2}}{p} \right) \left(a(a-1)+(a+p)(a+p-1)\right)\\
&= 2 \sum_{a=0}^{p-1} \left(\frac{a+\frac{p+1}{2}}{p} \right) (a^2+(a+p)^2) \\
&= 4 \sum_{a=0}^{p-1} \left(\frac{a+\frac{p+1}{2}}{p} \right)a^2 + 4p \sum_{a=0}^{p-1}  \left(\frac{a+\frac{p+1}{2}}{p} \right)a+ 2p^2 \sum_{a=0}^{p-1} \left(\frac{a+\frac{p+1}{2}}{p} \right) \\ 
&= 4 \sum_{a=0}^{p-1} \left(\frac{a+\frac{p+1}{2}}{p} \right)a^2
= 4\sum_{u=\frac{p+1}{2}}^{p} \left(\frac{u}{p}\right)(u-\frac{p+1}{2})^2 + 4 \sum_{u=1}^{\frac{p-1}{2}} \left(\frac{u}{p} \right)(u+\frac{p-1}{2})^2\\
&=4 \sum_{u=0}^{p-1} \left(\frac{u}{p} \right)u^2-4(p+1)\sum_{u=\frac{p+1}{2}}^{p}\left(\frac{u}{p} \right)u+4(p-1)\sum_{u=1}^{\frac{p-1}{2}}\left(\frac{u}{p} \right)u\\
&+(p+1)^2\sum_{u=\frac{p+1}{2}}^{p}\left(\frac{u}{p} \right)+(p-1)^2\sum_{u=1}^{\frac{p-1}{2}}\left(\frac{u}{p} \right)\\
&=4 \sum_{u=0}^{p-1} \left(\frac{u}{p} \right)u^2+4p(p+1)h(-p) -4ph(-p)\\
&=4 \sum_{u=0}^{p-1} \left(\frac{u}{p} \right)u^2+4p^2h(-p)
=0.
\end{align*} 
}
Here, we use (for $p\equiv 7\pmod 8$):
{\allowdisplaybreaks
\[
\begin{aligned}
\sum_{u=1}^{\frac{p-1}{2}}\left(\frac{u}{p} \right)u&=0\quad (\text{\cite[Corollary 13.4]{[Berndt]}}),
\sum_{u=1}^{p-1}\left(\frac{u}{p} \right)u=-ph(-p) \quad (\text{see \cite[Equation 3]{[Girstmair]}}),\\
\sum_{u=1}^{\frac{p-1}{2}}\left(\frac{u}{p} \right)&=h(-p) \quad (\text{\cite[Corollary 3.4]{[Berndt]}}),\\
\sum_{u=0}^{p-1} \left(\frac{u}{p} \right)u^2&=\frac{-\sqrt{p}p^2}{\pi}L(1,\chi_p)=-p^2h(-p) \; (\text{\cite[Theorem 13.15]{[Berndt]}}),
\end{aligned} .\]
}
Now we consider  $\mFe'''(-1)$. We have
{\allowdisplaybreaks
\begin{align*}
&\left(\frac{2}{p} \right) \mFe'''(-1) = 2 \sum_{a=0}^{p-1} \left(\frac{a+\frac{p+1}{2}}{p} \right) \left(-a(a-1)(a-2)-(a+p)(a+p-1)(a+p-2)\right)\\
&=- 2 \sum_{a=0}^{p-1} \left(\frac{a+\frac{p+1}{2}}{p} \right)(a^3+(a+p)^3)
= -4 \sum_{a=0}^{p-1} \left(\frac{a+\frac{p+1}{2}}{p} \right)a^3\\
&= -4\sum_{u=\frac{p+1}{2}}^{p} \left(\frac{u}{p}\right)(u-\frac{p+1}{2})^3 - 4 \sum_{u=1}^{\frac{p-1}{2}} \left(\frac{u}{p} \right)(u+\frac{p-1}{2})^3\\
&=-4 \sum_{u=0}^{p-1} \left(\frac{u}{p} \right)u^3+6(p+1)\sum_{u=\frac{p+1}{2}}^{p}\left(\frac{u}{p} \right)u^2-3(p+1)^2\sum_{u=\frac{p+1}{2}}^{p}\left(\frac{u}{p} \right)u+ \frac{(p+1)^3}{2}\sum_{u=\frac{p+1}{2}}^{p} \left(\frac{u}{p} \right)\\
&-6(p-1)\sum_{u=1}^{\frac{p-1}{2}}\left(\frac{u}{p} \right)u^2-3(p-1)^2\sum_{u=1}^{\frac{p-1}{2}}\left(\frac{u}{p} \right)u- \frac{(p-1)^3}{2}\sum_{u=1}^{\frac{p-1}{2}}\left(\frac{u}{p} \right).
\end{align*}
}

We have (for $p\equiv 7\pmod 8$)
\[
\begin{aligned}
L(3,\chi_p)&=\dfrac{i\sqrt{p}}{2i} \left(\dfrac{2\pi}{p} \right)^3 \dfrac{L(-2,\chi_p)}{\Gamma(3)\cos(\pi)} \quad\text{(see \cite[page 5]{[Iwasawa]})}\\
&=\dfrac{-2\sqrt{p}\pi^3}{p^3}L(-2,\chi_p)
=\dfrac{2\sqrt{p}\pi^3}{3p^3}B_{3,\chi} \quad\text{(see \cite[Theorem 1]{[Iwasawa]})}.
\end{aligned}
\]
Hence, by \cite[Theorem 13.15]{[Berndt]},
$
\sum\limits_{u=1}^{\frac{p-1}{2}} \left(\dfrac{u}{p}\right)u^2=\dfrac{-7\sqrt{p}p^2}{8\pi^3}
L(3,\chi_p)=-\dfrac{7}{12}B_{3,\chi},
$
and
$
\sum\limits_{u=\frac{p+1}{2}}^{p} \left(\dfrac{u}{p}\right)u^2=-p^2h(-p)+\dfrac{7}{12}B_{3,\chi}.
$
Thus 
\[
6(p+1)\sum_{u=\frac{p+1}{2}}^{p}\left(\frac{u}{p} \right)u^2-6(p-1)\sum_{u=1}^{\frac{p-1}{2}}\left(\frac{u}{p} \right)u^2=-6(p+1)p^2h(-p)+7pB_{3,\chi}. 
\]
One  also has
\[
\begin{aligned}
&-3(p+1)^2\sum_{u=\frac{p+1}{2}}^{p}\left(\frac{u}{p} \right)u+ \frac{(p+1)^3}{2}\sum_{u=\frac{p+1}{2}}^{p} \left(\frac{u}{p} \right)-3(p-1)^2\sum_{u=1}^{\frac{p-1}{2}}\left(\frac{u}{p} \right)u- \frac{(p-1)^3}{2}\sum_{u=1}^{\frac{p-1}{2}}\left(\frac{u}{p} \right)\\
&=3(p+1)^2ph(-p)-(p^3+3p)h(-p)=(2p^3+6p^2)h(-p).
\end{aligned}
\]
On the other hand, by \cite[page 10]{[Iwasawa]}
\[
\begin{aligned}
B_{3,\chi_p}&=p^2\sum_{u=1}^{p-1}\left(\dfrac{a}{p}\right)B_3(\frac{a}{p}-1)
=p^2\left[\sum_{u=1}^{p-1}\left(\dfrac{a}{p}\right)\frac{a^3}{p^3}-\dfrac{3}{2}\sum_{u=1}^{p-1}\left(\dfrac{a}{p}\right)\frac{a^2}{p^2}+\dfrac{1}{2}\sum_{u=1}^{p-1}\left(\dfrac{a}{p}\right)\frac{a}{p}\right].
\end{aligned}
\]
Thus
\[
\begin{aligned}
\sum_{u=1}^{p-1}\left(\dfrac{a}{p}\right)a^3&=pB_{3,\chi}+\dfrac{3}{2}p\sum_{u=1}^{p-1}\left(\dfrac{a}{p}\right)a^2-\dfrac{1}{2}p^2\sum_{u=1}^{p-1}\left(\dfrac{a}{p}\right)a\\
&=pB_{3,\chi_p}-\dfrac{3}{2}p^3h(-p)+\dfrac{1}{2}p^3h(-p)=pB_{3,\chi_p}-p^3h(-p).
\end{aligned}
\]
Therefore
\[
\begin{aligned}
\left(\frac{2}{p} \right) \mFe'''(-1)&=-4(pB_{3,\chi_p}-p^3h(-p))-6(p+1)p^2h(-p)+7pB_{3,\chi}+(2p^3+6p^2)h(-p)\\
&=3pB_{3,\chi}\not=0.
\end{aligned}
\qedhere
\]
\end{proof}

Surprisingly, the exceptional factor $\Phi_{12}(x)$ can happen in this case. 

\begin{prop}
$ r_{\Delta}(\Phi_{12})= \tilde{r}_{\Delta}(\Phi_6) = \begin{cases} 1 &\mbox{if } p \equiv 7 \pmod{8} \\
0 & \mbox{if } p \equiv 3 \pmod{8}. \end{cases} $  
\end{prop}

\begin{proof}
We have
\[
\begin{aligned}
\mFe(-\zeta_3)
&= \left(\frac{2}{p} \right) \sum_{a=0}^{p-1} \left(\frac{a+\frac{p+1}{2}}{p} \right)[(-1)^a (-\zeta_3)^a -(-1)^a (-\zeta_3)^{a+p}]\\
&=\begin{cases} -\left(\frac{2}{p} \right) \sum\limits_{a=0}^{p-1} \left(\frac{a+\frac{p+1}{2}}{p} \right)\zeta_3^{a+2} &\text{ if } p\equiv 1 \pmod 3\\
-\left(\frac{2}{p} \right) \sum\limits_{a=0}^{p-1} \left(\frac{a+\frac{p+1}{2}}{p} \right)\zeta_3^{a+1} &\text{ if } p\equiv 2 \pmod 3
\end{cases}\\
&=A+B\zeta_3+C\zeta_3^2.
\end{aligned}
\]
Note that $A+B+C=0$ and $1+\zeta_3+\zeta_3^2=0$. Hence  $\mFe(-\zeta_3)=0$ if and only if $A=B=C=0$.
\\
\\
\noindent{ \bf Case 1:} $p\equiv 3\pmod 8$ and $p\equiv 1\pmod 3$. In this case, $p$ is of the form $p=19+24k$, for some integer $k$ and
\[
\begin{aligned}
B&=-\left(\frac{2}{p} \right) \sum\limits_{\substack{a\equiv 2\mod 3\\ 0\leq a\leq p-1}} \left(\frac{a+\frac{p+1}{2}}{p} \right)
=-\left(\frac{2}{p} \right)\left(\frac{3}{p} \right) \sum\limits_{b=4+4k}^{9+12k} \left(\frac{b}{p} \right)\\
&=-\left(\frac{6}{p} \right)\left[\sum\limits_{0<b<p/2} \left(\frac{b}{p} \right)-\sum\limits_{0<b<p/6} \left(\frac{b}{p} \right) \right]
=-4h(-p).
\end{aligned}
\]
The last equality follows form the following facts (see \cite[Corollary 3.4]{[Berndt]})
\[
\sum\limits_{0<b<p/2} \left(\frac{b}{p} \right)=\left(2-\left(\dfrac{2}{p}\right)\right)h(-p)=3h(-p),
\]
and (see \cite[Theorem 6.1]{[Berndt]})
\[
\sum\limits_{0<b<p/6} \left(\frac{b}{p} \right)=\dfrac{h(-p)}{2}\left(1+\left(\dfrac{2}{p}\right)+\left(\dfrac{3}{p}\right)-\left(\dfrac{6}{p}\right)\right)=-h(-p).
\]
Thus $B\not=0$ and $\mFe(-\zeta_3)\not=0$.
\\
\\
\noindent{ \bf Case 2:} $p\equiv 3\pmod 8$ and $p\equiv 2\pmod 3$. In this case, 
we have
\[
\begin{aligned}
B&=-\left(\frac{2}{p} \right) \sum\limits_{\substack{a\equiv 0\mod 3\\ 0\leq a\leq p-1}} \left(\frac{a+\frac{p+1}{2}}{p} \right)
=-\left(\frac{6}{p} \right)\left[\sum\limits_{0<b<p/2} \left(\frac{b}{p} \right)-\sum\limits_{0<b<p/6} \left(\frac{b}{p} \right) \right]
\\&=2h(-p).
\end{aligned}
\]
The last equality follows form the facts (see \cite[Corollary 3.4]{[Berndt]})
\[
\sum\limits_{0<b<p/2} \left(\frac{b}{p} \right)=\left(2-\left(\dfrac{2}{p}\right)\right)h(-p)=3h(-p),
\]
and (see \cite[Theorem 6.1]{[Berndt]})
\[
\sum\limits_{0<b<p/6} \left(\frac{b}{p} \right)=\dfrac{h(-p)}{2}\left(1+\left(\dfrac{2}{p}\right)+\left(\dfrac{3}{p}\right)-\left(\dfrac{6}{p}\right)\right)=h(-p).
\]
Thus $B\not=0$ and $\mFe(-\zeta_3)\not=0$.
\\
\\
\noindent{ \bf Case 3:} $p\equiv 7\pmod 8$ and $p\equiv 1\pmod 3$. In this case, 
we have
\[
\begin{aligned}
B&=-\left(\frac{2}{p} \right) \sum\limits_{\substack{a\equiv 2\mod 3\\ 0\leq a\leq p-1}} \left(\frac{a+\frac{p+1}{2}}{p} \right)
&=-\left(\frac{6}{p} \right)\left[\sum\limits_{0<b<p/2} \left(\frac{b}{p} \right)-\sum\limits_{0<b<p/6} \left(\frac{b}{p} \right) \right]
=0.
\end{aligned}
\]
The last equality follows form the facts (see \cite[Corollary 3.4]{[Berndt]})
\[
\sum\limits_{0<b<p/2} \left(\frac{b}{p} \right)=\left(2-\left(\dfrac{2}{p}\right)\right)h(-p)=h(-p),
\]
and (see \cite[Theorem 6.1]{[Berndt]})
\[
\sum\limits_{0<b<p/6} \left(\frac{b}{p} \right)=\dfrac{h(-p)}{2}\left(1+\left(\dfrac{2}{p}\right)+\left(\dfrac{3}{p}\right)-\left(\dfrac{6}{p}\right)\right)=h(-p).
\]
We also have
\[
\begin{aligned}
C&=-\left(\frac{2}{p} \right) \sum\limits_{\substack{a\equiv 0\mod 3\\ 0\leq a\leq p-1}} \left(\frac{a+\frac{p+1}{2}}{p} \right)
=-\sum\limits_{\substack{b\equiv 1\mod 3\\ p/2< b<p }} \left(\frac{b}{p} \right)-\sum\limits_{\substack{b\equiv 1\mod 3\\ p< b<3p/2}} \left(\frac{b}{p} \right)\\
&=\sum\limits_{\substack{b\equiv 1\mod 3\\ p/2<b<p}} \left(\frac{p-b}{p} \right)-\sum\limits_{\substack{b\equiv 1\mod 3\\ p< b< 3p/2}} \left(\frac{b-p}{p} \right)
=\sum\limits_{\substack{b\equiv 0\mod 3\\ 0<b<p/2}} \left(\frac{b}{p} \right)-\sum\limits_{\substack{b\equiv 0\mod 3\\ 0<b< p/2}} \left(\frac{b}{p} \right)=0. 
\end{aligned}
\]
Hence $C=0$. This implies that $A=0$ and $\mFe(-\zeta_3)=0$.

On the other hand, we have
\begin{align*}
\left(\frac{2}{p} \right) \mFe'(-\zeta_3) &= \sum_{a=0}^{p-1}     \left(\frac{a+\frac{p+1}{2}}{p} \right) (-a\zeta_3^{a-1} -(a+p)\zeta_3^{a+p-1}) \\
&= - \sum_{a=0}^{p-1}\left(\frac{a+\frac{p+1}{2}}{p} \right)a(\zeta_3^{a-1}+\zeta_3^a) -p \sum_{a=0}^{p-1} \left(\frac{a+\frac{p+1}{2}}{p} \right)\zeta_3^a \\
&=\sum_{a=0}^{p-1}\left(\frac{a+\frac{p+1}{2}}{p} \right)a\zeta_3^{a+1} -p \sum_{a=0}^{p-1} \left(\frac{a+\frac{p+1}{2}}{p} \right)\zeta_3^a
= E+F\zeta_3+G\zeta_3^2,
 \end{align*}
 here
 \[
 \begin{aligned}
 E&=\sum_{\substack{0\leq a\leq p-1\\ a\equiv 2\mod 3}} \left(\frac{a+\frac{p+1}{2}}{p} \right)a -p\sum_{\substack{0\leq a\leq p-1\\ a\equiv 0\mod 3}} \left(\frac{a+\frac{p+1}{2}}{p} \right)=\sum_{\substack{0\leq a\leq p-1\\ a\equiv 2\mod 3}} \left(\frac{a+\frac{p+1}{2}}{p} \right)a,\\
  F&=\sum_{\substack{0\leq a\leq p-1\\ a\equiv 0\mod 3}} \left(\frac{a+\frac{p+1}{2}}{p} \right)a -p\sum_{\substack{0\leq a\leq p-1\\ a\equiv 1\mod 3}} \left(\frac{a+\frac{p+1}{2}}{p} \right)=\sum_{\substack{0\leq a\leq p-1\\ a\equiv 0\mod 3}} \left(\frac{a+\frac{p+1}{2}}{p} \right)a,\\
G&=\sum_{\substack{0\leq a\leq p-1\\ a\equiv 1\mod 3}} \left(\frac{a+\frac{p+1}{2}}{p} \right)a -p\sum_{\substack{0\leq a\leq p-1\\ a\equiv 2\mod 3}} \left(\frac{a+\frac{p+1}{2}}{p} \right)=\sum_{\substack{0\leq a\leq p-1\\ a\equiv 1\mod 3}} \left(\frac{a+\frac{p+1}{2}}{p} \right)a.
 \end{aligned}
  \]
  Note that 
  $
   E+F+G=\sum\limits_{0\leq a\leq p-1} \left(\dfrac{a+\frac{p+1}{2}}{p} \right)a =-\dfrac{1}{2}\left(\dfrac{2}{p}\right)\mFe'(-1)=0.
$ 
  Hence $\mFe'(-\zeta_3)=0$ if and only if $E=F=G=0$.

 By considering modulo $2$, we have
  \[
  \begin{aligned}
  G&\equiv \sum_{\substack{0\leq a\leq p-1\\ a\equiv 1\mod 3\\a\equiv 1\pmod 2}} \left(\frac{a+\frac{p+1}{2}}{p} \right)
  \equiv \sum_{\substack{0\leq a\leq p-1\\ a\equiv 1\mod 6}} 1 = \dfrac{p-1}{6} \equiv 1\pmod 2.
    \end{aligned}
      \]
 In particular, $G\not=0$ and thus $\mFe'(-\zeta_3)\not=0$.
\\
\\
\noindent{ \bf Case 4:} $p\equiv 7\pmod 8$ and $p\equiv 2\pmod 3$. In this case, 
we have
\[
\begin{aligned}
B&=-\left(\frac{2}{p} \right) \sum\limits_{\substack{a\equiv 0\mod 3\\ 0\leq a\leq p-1}} \left(\frac{a+\frac{p+1}{2}}{p} \right)
&=-\left(\frac{6}{p} \right)\left[\sum\limits_{0<b<p/2} \left(\frac{b}{p} \right)-\sum\limits_{0<b<p/6} \left(\frac{b}{p} \right) \right]
=0.
\end{aligned}
\vspace{-.1cm}
\]
The last equality follows from the facts (see \cite[Corollary 3.4]{[Berndt]})
\[
\sum\limits_{0<b<p/2} \left(\frac{b}{p} \right)=\left(2-\left(\dfrac{2}{p}\right)\right)h(-p)=h(-p)
\]
and (see \cite[Theorem 6.1]{[Berndt]})
\[
\sum\limits_{0<b<p/6} \left(\frac{b}{p} \right)=\dfrac{h(-p)}{2}\left(1+\left(\dfrac{2}{p}\right)+\left(\dfrac{3}{p}\right)-\left(\dfrac{6}{p}\right)\right)=h(-p).
\]
We also have
\[
\begin{aligned}
A&=-\left(\frac{2}{p} \right) \sum\limits_{\substack{a\equiv 2\mod 3\\ 0\leq a\leq p-1}} \left(\frac{a+\frac{p+1}{2}}{p} \right)
=-\sum\limits_{\substack{b\equiv 2\mod 3\\ p/2< b<p }} \left(\frac{b}{p} \right)-\sum\limits_{\substack{b\equiv 2\mod 3\\ p< b<3p/2}} \left(\frac{b}{p} \right)\\
&=\sum\limits_{\substack{b\equiv 2\mod 3\\ p/2<b<p}} \left(\frac{p-b}{p} \right)-\sum\limits_{\substack{b\equiv 2\mod 3\\ p\leq b\leq 3p/2}} \left(\frac{b-p}{p} \right)
=\sum\limits_{\substack{b\equiv 0\mod 3\\ 0<b<p/2}} \left(\frac{b}{p} \right)-\sum\limits_{\substack{b\equiv 0\mod 3\\ 0\leq b\leq p/2}} \left(\frac{b}{p} \right)=0.
\end{aligned}
\]
 This implies that $C=0$ and $\mFe(-\zeta_3)=0$.

 On the other hand, we have
\begin{align*}
\left(\frac{2}{p} \right) \mFe'(-\zeta_3) &= \sum_{a=0}^{p-1}     \left(\frac{a+\frac{p+1}{2}}{p} \right) (-a\zeta_3^{a-1} -(a+p)\zeta_3^{a+p-1}) \\
&= - \sum_{a=0}^{p-1}\left(\frac{a+\frac{p+1}{2}}{p} \right)a(\zeta_3^{a-1}+\zeta_3^{a+1}) -p \sum_{a=0}^{p-1} \left(\frac{a+\frac{p+1}{2}}{p} \right)\zeta_3^{a+1} \\
&=\sum_{a=0}^{p-1}\left(\frac{a+\frac{p+1}{2}}{p} \right)a\zeta_3^{a} -p \sum_{a=0}^{p-1} \left(\frac{a+\frac{p+1}{2}}{p} \right)\zeta_3^{a+1}
= E+F\zeta_3+G\zeta_3^2,
 \end{align*}
 here
 {\allowdisplaybreaks
  \begin{align*}
 E
 &=\sum_{\substack{0\leq a\leq p-1\\ a\equiv 0\mod 3}} \left(\frac{a+\frac{p+1}{2}}{p} \right)a,
  \quad\quad F
  =\sum_{\substack{0\leq a\leq p-1\\ a\equiv 1\mod 3}} \left(\frac{a+\frac{p+1}{2}}{p} \right)a,\\
G&
=\sum_{\substack{0\leq a\leq p-1\\ a\equiv 2\mod 3}} \left(\frac{a+\frac{p+1}{2}}{p} \right)a.
 \end{align*}
  }
  
We have
\vspace{-.1cm}
\begin{align*} E &=\sum_{\substack{0\leq a\leq p-1\\ a\equiv 0\mod 3}} \left(\frac{a+\frac{p+1}{2}}{p} \right)a 
= 3 \sum_{0 \leq m < p/3} \left(\frac{m+\frac{p+1}{6}}{p} \right) m \\
&= 3 \sum_{p/6 <u <p/2} \left(\frac{u}{p} \right)u - \frac{p+1}{6} \sum_{p/6<u<p/2} \left(\frac{u}{p} \right).
\end{align*}
By  \cite[Page 23]{[MTT3]},
$\sum\limits_{p/6<u<p/2} \left(\dfrac{u}{p} \right) =S_{12}-S_{16}= h(-p)-h(-p)=0.$ 
Therefore 
$E = 3 \sum\limits_{p/6 <u <p/2} \left(\dfrac{u}{p} \right)u.$  

Recall from \cite[Section 13]{[Berndt]} that 
$ S_{ij}(\chi_p,1) = \sum\limits_{(i-1)/j< p < pi/ j} \chi_p(a) a. $
By \cite[Theorem 13.3]{[Berndt]} and the fact that  $p \equiv 7 \pmod{8}$, $S_{12}(\chi_p,1) = 0$ and hence 
$ E+ 3S_{16}(\chi_p, 1) =0.$

By \cite[Formula 12.2]{[Berndt]}, applying for $\chi=\chi_{-p}$, $f(x)=x$, $c=0$ and $d=p/6$, we have
 \[
 \begin{aligned}
 S_{16}(\chi,1)&=-\dfrac{2iG(\chi)}{p}\sum\limits_{n=1}^\infty\chi(n) \int\limits_c^d x\sin(\dfrac{2\pi n x}{p})dx\\
 &=\dfrac{2iG(\chi)}{p}\sum\limits_{n=1}^\infty\chi(n)\dfrac{p^2}{12\pi^2n^2}\left(\pi n\cos(\dfrac{\pi n}{3})-3\sin(\dfrac{\pi n}{3})\right)
 \end{aligned}
 \]
  We have 
  {\allowdisplaybreaks
  \begin{align*}
 \sum_{n=1}^\infty \dfrac{\chi(n)}{n}\cos(\pi n/3)&=\dfrac{1}{2}\sum_{\substack{1\leq n\\n\equiv 1\mod 6}}\dfrac{\chi(n)}{n}-\dfrac{1}{2}\sum_{\substack{1\leq n\\n\equiv 2\mod 6}}\dfrac{\chi(n)}{n}-\sum_{\substack{1\leq n\\n\equiv 3\pmod 6}}\dfrac{\chi(n)}{n}\\
 &-\dfrac{1}{2}\sum_{\substack{4\leq n\\n\equiv 4\pmod 6}}\dfrac{\chi(n)}{n}+\dfrac{1}{2}\sum_{\substack{1\leq n\\n\equiv 5\pmod 6}}\dfrac{\chi(n)}{n}+\sum_{\substack{1\leq n\\n\equiv 0\pmod 6}}\dfrac{\chi(n)}{n}\\
 &=\dfrac{1}{2}\sum_{1\leq n}\dfrac{\chi(n)}{n} -\sum_{\substack{1\leq n\\n\equiv 0\pmod 2}}\dfrac{\chi(n)}{n}\\
 &-\dfrac{3}{2}\sum_{\substack{1\leq n\\n\equiv 0\pmod 3}}\dfrac{\chi(n)}{n}+3\sum_{\substack{1\leq n\\n\equiv 0\pmod 6}}\dfrac{\chi(n)}{n}\\
 &=\dfrac{1}{2}L(\chi,1)-\dfrac{\chi(2)}{2}L(\chi,1)-\dfrac{3}{2}\dfrac{\chi(3)}{3}L(\chi,1)+3\dfrac{\chi(6)}{6}L(\chi,1)
 =0.
 \end{align*}
 }
 Hence 
 \[
 \begin{aligned}
 S_{16}(\chi,1)&=-\dfrac{2iG(\chi)}{p}\dfrac{p^2}{12\pi^2n^2}\times 3\sum_{n=1}^\infty  \chi(n)\sin(\pi n/3)
 =\dfrac{p\sqrt{p}}{2\pi^2}\sum_{n=1}^\infty \dfrac{\chi(n)}{n^2}\sin(\pi n/3).
 \end{aligned}
 \]
 Note that
 \[
 \begin{aligned}
 \sum_{n=1}^\infty \dfrac{\chi(n)}{n^2}\sin(\pi n/3)&=\dfrac{1}{2}\left( \dfrac{\chi(1)}{1}+\dfrac{\chi(2)}{2^2}-\dfrac{\chi(4)}{4^2}-\dfrac{\chi(5)}{5^2}+\cdots\right)\\
 &\geq \dfrac{\sqrt{3}}{2}\left(1+\dfrac{1}{4}-\sum_{n\geq 4}\dfrac{1}{n^2}\right)=\dfrac{\sqrt{3}}{2}\left(1+\dfrac{1}{4}-\dfrac{\pi^2}{6}+1+\dfrac{1}{2^2}+\dfrac{1}{3^2}\right)>0.
 \end{aligned}
 \]
 Therefore $E<0$ and thus $F'_\Delta(-\zeta_3)\not=0$.
  \end{proof}

\subsection{The Fekete polynomials $f_{\Delta}(x)$ and $g_{\Delta}(x)$}
From the results discussed in the previous section, we introduce the following definition. 
\begin{definition}
Let $\Delta = 4p$ with $p \equiv 3 \pmod{4}$. The Fekete polynomial $f_{\Delta}(x)$ is given by the following formula
\[ f_{\Delta}(x)= \begin{cases} \dfrac{\mFe(x)}{\Phi_1(x)^2 \Phi_2(x) \Phi_p(x)} &\mbox{if $p \equiv 3 \pmod{8}$} \\
\dfrac{\mFe(x)}{\Phi_1(x)^2 \Phi_2(x)^3 \Phi_6(x) \Phi_p(x)} & \mbox{if $p \equiv 7 \pmod{8}$}. \end{cases} \] 
\end{definition}

We have the following observation. 
\begin{prop}
$f_{\Delta}(x)$ is a reciprocal polynomial of degree 
\[ \deg(f_{\Delta})= \begin{cases} p-7 &\mbox{if } p \equiv 7 \pmod{8} \\
p-3 & \mbox{if } p \equiv 3 \pmod{8}. \end{cases} \] 
\end{prop}
\begin{proof}
By definition $\mFe(x)$ is a polynomial of degree $D/2-1 = 2p-1.$ The degree of $f_{\Delta}$ described above follows from its definition. Furthermore, by Proposition \ref{prop:reciprocal1}, we know that $\mFe(x)$ is a reciprocal polynomial. Additionally, by Lemma \ref{lem:reciprocal}, we also know that $\Phi_n(x)$ is a reciprocal polynomial if $n \geq 2.$ We conclude that $f_{\Delta}$ is a reciprocal polynomial as well. 

\end{proof}
We then define the reduced Fekete polynomial $g_{\Delta}(x)$ as the trace polynomial of $f_{\Delta}(x).$ 
\begin{definition}
The reduced Fekete polynomial $g_{\Delta}(x) \in \Z[x]$ is defined as 
\[ f_{\Delta}(x) = x^{\frac{\deg(f_{\Delta})}{2}} g_{\Delta}(x+\frac{1}{x}) .\] 
\end{definition}

\section{Generalized Fekete polynomials with $\Delta=-4p$, $p \equiv 1 \pmod{4}$} 
\label{sec:-4p}
\subsection{Cyclotomic factors of $\mFe(x)$}
Similar to the case $\Delta= 4p$, we know the following equality (see Corollary~\ref{cor:Even_Multiplicity_Phi1_Phi2})
\[ r_{\Delta}(\Phi_1)= r_{\Delta}(\Phi_2)= \tilde{r}_{\Delta}(\Phi_1) = 1 .\] 
We next determine the multiplicity of $\Phi_p(x), \Phi_{2p}(x)$ in $F_{\Delta}(x).$ We first recall the definition of $\mFe(x)$
\[ \mFe(x) = \sum_{a=0}^{2p-1} \left(\frac{-4p}{2a+1} \right) x^a .\] 
By the reciprocity law, we have 
\[ \left(\frac{-4p}{2a+1} \right)=\left(\frac{-1}{2a+1} \right) \left(\frac{p}{2a+1} \right)= (-1)^a \left(\frac{2a+1}{p} \right) .\]
Wen can then deduce that 
\[ \mFe(x) = \left(\frac{2}{p} \right) \sum_{a=0}^{p-1} \left(\frac{a+\frac{p+1}{2}}{p} \right) \left [(-1)^a x^a -(-1)^a x^{a+p} \right] .\]


\begin{prop}
$  r_{\Delta}(\Phi_p)=r_{\Delta}(\Phi_{2p})= \tilde{r}_{\Delta}(\Phi_p) =1.$ 
\end{prop}

\begin{proof}
We have 
$
\mFe'(x)=\left(\dfrac{2}{p} \right) \sum\limits_{a=1}^{p-1} \left(\dfrac{a+\frac{p+1}{2}}{p} \right)[(-1)^a a x^{a-1} -(-1)^a (a+p) x^{a+p-1}].
$
Hence
\[
\begin{aligned}
\mFe'(\zeta_p)
=-\left(\frac{2}{p} \right)p \sum_{a=1}^{p-1} \left(\frac{a+\frac{p+1}{2}}{p} \right)(-1)^a  \zeta_p^{a-1}.
\end{aligned}
\]
By Lemma \ref{lem:cyclotomic_basis_square_free} we know that $\{\zeta_p, \zeta_p^2, \ldots, \zeta_p^{p-1} \}$ is $\Q$-linearly independent. Therefore, we can conclude that $\mFe'(\zeta_p) \neq 0$ and the above statement holds. 
\end{proof}

Regarding the multiplicity of $\Phi_4(x)$, we have the following 

\begin{prop}
$  r_{\Delta}(\Phi_4)= \tilde{r}_{\Delta}(\Phi_2) = 2. $ 
\end{prop}

\begin{proof}
We need to show that $\mFe(-1)=\mFe'(-1)=0$ and $\mFe''(-1) \neq 0.$
We have 
\[ \mFe(-1) = 2 \left(\frac{2}{p} \right) \sum_{a=0}^{p-1} \left(\frac{a+\frac{p+1}{2}}{p} \right)=0 .\] 
We also have 
\begin{align*}
\left(\frac{2}{p} \right) \mFe'(-1) &= \sum_{a=0}^{p-1}     \left(\frac{a+\frac{p+1}{2}}{p} \right) (-2a -p) 
= - 2 \sum_{a=0}^{\frac{p-1}{2}} \left(\frac{a+\frac{p+1}{2}}{p} \right)a - 2 \sum_{a=\frac{p+1}{2}}^{p-1} \left(\frac{a+\frac{p+1}{2}}{p} \right)a \\
&=- 2 \sum_{u=\frac{p+1}{2}}^{p} \left(\frac{u}{p} \right)(u-\frac{p+1}{2}) - 2 \sum_{u=1}^{\frac{p-1}{2}} \left(\frac{u}{p} \right)(u+\frac{p-1}{2}) 
= - 2 \sum_{u=1}^{p-1} \left(\frac{u}{p} \right)u  =0. 
\end{align*}
For the second to the last equality we use the fact that for $p \equiv 1 \pmod{4}$ (see \cite[Lemma 2.5]{[MTT3]}), 
$\sum\limits_{u=\frac{p+1}{2}}^p  \left(\dfrac{u}{p} \right) = \sum\limits_{u=1}^{\frac{p-1}{2}} \left(\dfrac{u}{p} \right)= 0.$

Finally, we have 
{ \allowdisplaybreaks
\begin{align*}
\left(\frac{2}{p} \right) \mFe''(-1) &= 2 \sum_{a=0}^{p-1} \left(\frac{a+\frac{p+1}{2}}{p} \right) (a^2+(a+p)^2) \\
&= 4 \sum_{a=0}^{p-1} \left(\frac{a+\frac{p+1}{2}}{p} \right)a^2 + 4p \sum_{a=0}^{p-1}  \left(\frac{a+\frac{p+1}{2}}{p} \right)a+ 2p^2 \sum_{a=0}^{p-1} \left(\frac{a+\frac{p+1}{2}}{p} \right) \\ 
&= 4 \sum_{a=0}^{p-1} \left(\frac{a+\frac{p+1}{2}}{p} \right)a^2
=4 \sum_{u=0}^{p-1} \left(\frac{u}{p} \right)\left(u-\frac{p+1}{2}\right)^2
=4 \sum_{u=0}^{p-1} \left(\frac{u}{p} \right)u^2\not=0. 
\end{align*} 
}
The last equation follows from \cite[Lemma 2.2]{[MTT3]}).
\end{proof}

For the exceptional factor $\Phi_{12}(x)$, we have 

\begin{prop}
$  r_{\Delta}(\Phi_{12})= \tilde{r}_{\Delta}(\Phi_6) = \begin{cases} 0 &\mbox{if } p \equiv 1 \pmod{8} \\
1 & \mbox{if } p \equiv 5 \pmod{8}. \end{cases}$ 
\end{prop}

\begin{proof}
We have
\[
\begin{aligned}
\mFe(-\zeta_3)
&=\left(\frac{2}{p} \right) \sum_{a=0}^{p-1} \left(\frac{a+\frac{p+1}{2}}{p} \right) [\zeta_3^a+\zeta_3^{a+p}]\\
&=\begin{cases} -\left(\frac{2}{p} \right) \sum\limits_{a=0}^{p-1} \left(\frac{a+\frac{p+1}{2}}{p} \right)\zeta_3^{a+2} &\text{ if } p\equiv 1 \pmod 3\\
-\left(\frac{2}{p} \right) \sum\limits_{a=0}^{p-1} \left(\frac{a+\frac{p+1}{2}}{p} \right)\zeta_3^{a+1} &\text{ if } p\equiv 2 \pmod 3
\end{cases}\\
&=A+B\zeta_3+C\zeta_3^2.
\end{aligned}
\]
Note that $A+B+C=0$ and $1+\zeta_3+\zeta_3^2=0$. Hence  $\mFe(-\zeta_3)=0$ if and only if $A=B=C=0$.
\\
\\
\noindent{ \bf Case 1:} $p\equiv 1\pmod 8$ and $p\equiv 1\pmod 3$. In this case, $B=h(-3p)\not=0$ and $\mFe(-\zeta_3)\not=0$.

\noindent{ \bf Case 2:} $p\equiv 1\pmod 8$ and $p\equiv 2\pmod 3$. In this case,
$B=-h(-3p)\not=0$ and $\mFe(-\zeta_3)\not=0$.

\noindent{ \bf Case 3:} $p\equiv 5\pmod 8$ and $p\equiv 1\pmod 3$. In this case, 
$B=C=0=A$ and $\mFe(-\zeta_3)=0$.

On the other hand, we have
\begin{align*}
\left(\frac{2}{p} \right) \mFe'(-\zeta_3) &= \sum_{a=0}^{p-1}     \left(\frac{a+\frac{p+1}{2}}{p} \right) (-a\zeta_3^{a-1} -(a+p)\zeta_3^{a+p-1}) \\
&=\sum_{a=0}^{p-1}\left(\frac{a+\frac{p+1}{2}}{p} \right)a\zeta_3^{a+1} -p \sum_{a=0}^{p-1} \left(\frac{a+\frac{p+1}{2}}{p} \right)\zeta_3^a
= E+F\zeta_3+G\zeta_3^2,
 \end{align*}
 here
 \[
 \begin{aligned}
 E&=
 =\sum_{\substack{0\leq a\leq p-1\\ a\equiv 2\mod 3}} \left(\frac{a+\frac{p+1}{2}}{p} \right)a,\quad\quad
  F
  =\sum_{\substack{0\leq a\leq p-1\\ a\equiv 0\mod 3}}\left(\frac{a+\frac{p+1}{2}}{p} \right)a,\\
G&
=\sum_{\substack{0\leq a\leq p-1\\ a\equiv 1\mod 3}} \left(\frac{a+\frac{p+1}{2}}{p} \right)a.
 \end{aligned}
  \]

  One can check that $E=-3\dfrac{p\sqrt{p}}{\pi^2}L(\chi_p,2)<0$ and hence $F'_\Delta(-\zeta_3)\not=0$.

\noindent{ \bf Case 4:} $p\equiv 5\pmod 8$ and $p\equiv 2\pmod 3$. 
In this case, 
$B=C=0=A$ and $\mFe(-\zeta_3)=0$.

 On the other hand, we have
\begin{align*}
\left(\frac{2}{p} \right) \mFe'(-\zeta_3) &= \sum_{a=0}^{p-1}     \left(\frac{a+\frac{p+1}{2}}{p} \right) (-a\zeta_3^{a-1} -(a+p)\zeta_3^{a+p-1}) \\
&=\sum_{a=0}^{p-1}\left(\frac{a+\frac{p+1}{2}}{p} \right)a\zeta_3^{a} -p \sum_{a=0}^{p-1} \left(\frac{a+\frac{p+1}{2}}{p} \right)\zeta_3^{a+1}
= E+F\zeta_3+G\zeta_3^2,
 \end{align*}
 here
 \[
 \begin{aligned}
 E&
 =\sum_{\substack{0\leq a\leq p-1\\ a\equiv 0\mod 3}} \left(\frac{a+\frac{p+1}{2}}{p} \right)a,\quad\quad
  F
  =\sum_{\substack{0\leq a\leq p-1\\ a\equiv 1\mod 3}} \left(\frac{a+\frac{p+1}{2}}{p} \right)a,\\
G&
=\sum_{\substack{0\leq a\leq p-1\\ a\equiv 2\mod 3}} \left(\frac{a+\frac{p+1}{2}}{p} \right)a.
 \end{aligned}
  \]
  
   By considering modulo $2$, we have
  \[
  \begin{aligned}
  F&\equiv \sum_{\substack{0\leq a\leq p-1\\ a\equiv 1\mod 3\\a\equiv 1\mod 2}} \left(\frac{a+\frac{p+1}{2}}{p} \right) 
  \equiv \sum_{\substack{0\leq a\leq p-1\\ a\equiv 1\mod 6}} 1 
  \equiv \dfrac{p-1}{6} \equiv 1\pmod 2.
    \end{aligned}
      \]
   In particular, $F\not=0$ and thus $\mFe'(-\zeta_3)\not=0$.
\end{proof}

\subsection{The Fekete polynomials $f_{\Delta}(x)$ and $g_{\Delta}(x)$}
We introduce the following definition. 
\begin{definition}
Let $\Delta =-4p$ with $p \equiv 1 \pmod{4}.$ The Fekete polynomial $f_{\Delta}(x)$ is given by the following formula 
\[ f_{\Delta}(x)= \begin{cases} \dfrac{-\mFe(x)}{\Phi_1(x) \Phi_2(x)^2 \Phi_6(x) \Phi_p(x) } &\mbox{if $p \equiv 5 \pmod{8}$} \\
\dfrac{-\mFe(x)}{\Phi_1(x) \Phi_2(x)^2 \Phi_p(x)} & \mbox{if $p \equiv 1 \pmod{8}$}. \end{cases} \] 
\end{definition}
We have the following observation. 
\begin{prop}
$f_{\Delta}(x)$ is a reciprocal polynomial of degree 
\[ \deg(f_{\Delta})= \begin{cases} p-5 &\mbox{if } p \equiv 5 \pmod{8} \\
p-3 & \mbox{if } p \equiv 1 \pmod{8}. \end{cases} \] 
\end{prop}
\begin{proof}
This follows from Lemma \ref{lem:reciprocal} and Proposition \ref{prop:reciprocal1}. 
\end{proof}

As before, we define $g_{\Delta}(x) \in \Z[x]$ to be the trace polynomial of $f_{\Delta}(x)$; namely 

\[ f_{\Delta}(x) = x^{\frac{\deg(f_{\Delta})}{2}} g_{\Delta}(x+\frac{1}{x}) .\]  


\section{ Generalized Fekete polynomial with $\Delta =3p$, $p \equiv 3 \pmod{4}$}
\label{sec:3p}
In this section, we discuss the case $\Delta=3p$ with $p \equiv 3 \pmod{4}.$ We first have the following simple observation. 
\begin{prop}
$r_\Delta(\Phi_p)=1$.
\end{prop}
\begin{proof} This follows from Proposition~\ref{prop:positive_big_divisor}
\end{proof}

\begin{prop} \label{prop:3p:zeta_3}
Let $p\equiv 3 \pmod{4}$ be a prime number.
\begin{enumerate}
\item[(a)] The number $\zeta_3$ is a double root of $F_\Delta$ if $p\equiv 2\pmod 3$ and a simple root if $p\equiv 1\pmod 3$. In other words,
$
r_\Delta(\Phi_3)=\begin{cases}
2 &\text{ if } p\equiv 2\pmod 3\\
1 &\text{ if } p\equiv 1\pmod 3.
\end{cases}
$
\item[(b)] The number $-\zeta_3$ is a simple root of $F_\Delta$ if $p\equiv 2\pmod 3$ and it is not a root if $p\equiv 1\pmod 3$. In other words,
$
r_\Delta(\Phi_6)=\begin{cases}
1 &\text{ if } p\equiv 2\pmod 3\\
0 &\text{ if } p\equiv 1\pmod 3.
\end{cases}
$
\end{enumerate}
\end{prop}
\begin{proof}
(a) One has 
$F'_\Delta(x)=\sum\limits_{k=1}^{p}(3k-2)\left(\dfrac{3p}{3k-2}\right)x^{3k-3}+\sum\limits_{k=1}^{p}(3k-1)\left(\dfrac{3p}{3k-1}\right)x^{3k-2}.$
Hence 
\[
F'_\Delta(\zeta_3)=\sum\limits_{k=1}^{p}(3k-2)\left(\dfrac{3p}{3k-2}\right)+\zeta_3\sum\limits_{k=1}^{p}(3k-1)\left(\dfrac{3p}{3k-1}\right)=A+B\zeta_3, 
\]
where
\[
\begin{aligned}
A&=\sum\limits_{k=1}^{p}(3k-2)\left(\dfrac{3p}{3k-2}\right)=\sum\limits_{k=1}^{p}(3k-2)\left(\dfrac{3k-2}{p}\right)\\
B&=\sum\limits_{k=1}^{p}(3k-1)\left(\dfrac{3p}{3k-1}\right)=\sum\limits_{k=1}^{p}(3k-1)\left(\dfrac{3k-1}{p}\right).
\end{aligned}
\]
Suppose that $p\equiv 1\pmod 3$. Then $p=12n+7$ for some integer $n$. Note that $3k-2$ is odd if and only if $k=2l+1$ for some $l$. In this case $3k-2=6l+1$. Hence 
\[
A\equiv \sum\limits_{l=0}^{6n+3} \left(\dfrac{6l+1}{p}\right)
\pmod 2.
\]
 Note also that the set $\{6l+1 \}_{l=0}^{6n+3}$ contains exactly one element which is divisible by $p$, namely the element $6(2n+1)+1=p$. Hence the set $\{\left(\frac{6l+1}{p} \right) \}_{l=0}^{6n+3}$ has exactly one element equal to $0$. The other elements must be $\pm{1}$. Thus
 \[
 A\equiv \sum\limits_{l=0}^{6n+3}\left(\dfrac{6l+1}{p}\right)\equiv 6n+3\equiv 1\pmod 2.
 \]
 Therefore $A\not=0$ and $F'_\Delta(\zeta_3)\not=0$. 
 
 Now suppose that $p\equiv 2\pmod 3$. Then 
 \[
\begin{aligned} A&=\sum_{\stackrel{a\equiv 1\pmod 3}{0<a<p}} a\left(\dfrac{a}{p}\right) + \sum_{\stackrel{a\equiv 1\pmod 3}{p<a<2p}} a\left(\dfrac{a}{p}\right)+\sum_{\stackrel{a\equiv 1\pmod 3}{2p<a<3p}} a\left(\dfrac{a}{p}\right)\\
&=\sum_{\stackrel{a\equiv 1\pmod 3}{0<a<p}} a\left(\dfrac{a}{p}\right) + \sum_{\stackrel{a\equiv 2\pmod 3}{0<a<p}} (p+a)\left(\dfrac{p+a}{p}\right)+\sum_{\stackrel{a\equiv 0\pmod 3}{0<a<p}} (2p+a)\left(\dfrac{2p+a}{p}\right)\\
&=\sum\limits_{a=1}^{p-1}a\left(\dfrac{a}{p}\right)+p\left[\sum_{\stackrel{a\equiv 2\pmod 3}{0<a<p}} \left(\dfrac{a}{p}\right) +\sum_{\stackrel{a\equiv 0\pmod 3}{0<a<p}} \left(\dfrac{a}{p}\right)\right]+p\sum_{\stackrel{a\equiv 0\pmod 3}{0<a<p}} \left(\dfrac{a}{p}\right).
\end{aligned}
 \]
 One has
 $
  \sum\limits_{a=1}^{p-1}a\left(\dfrac{a}{p}\right)=-ph(-p) \quad  (\text{see \cite[Equation 3]{[Girstmair]}}),
 $
 and
 \[
 \sum_{\stackrel{a\equiv 2\pmod 3}{0<a<p}} \left(\dfrac{a}{p}\right) +\sum_{\stackrel{a\equiv 0\pmod 3}{0<a<p}} \left(\dfrac{a}{p}\right)=
 \sum_{\stackrel{a\equiv 2\pmod 3}{0<a<p}} \left(\dfrac{a}{p}\right) +\sum_{\stackrel{a\equiv 2\pmod 3}{0<a<p}} \left(\dfrac{p-a}{p}\right)=0,
 \]
 and
\[
\begin{aligned}\sum_{\stackrel{a\equiv 0\pmod 3}{0<a<p}} \left(\dfrac{a}{p}\right)&=\left(\dfrac{3}{p}\right)\sum\limits_{0<a<p/3}\left(\dfrac{a}{p}\right)
=\dfrac{1}{2}\left(3-\left(\dfrac{3}{p}\right)\right)h(-p)=h(-p) \;(\text{\cite[Corollary 4.3]{[Berndt]}}).
\end{aligned}
\]
 Hence $A=0$. On the other hand, $F'_\Delta(1)=A+B=0$. Thus $B=0$ and $F'_\Delta(\zeta_3)=0$.
 Now we are going to show that $F''_\Delta(\zeta_3)\not=0$. 
  We have
  \[
 F''_\Delta(x)
 =\sum\limits_{a=2}^{p}(3a-2)(3a-3)\left(\dfrac{3p}{3a-2}\right)x^{3a-4}+\sum\limits_{a=1}^{p}(3a-1)(3a-2)\left(\dfrac{3p}{3a-1}\right)x^{3a-3}.
 \]
 Hence $F''_\Delta(\zeta_3)$ is of the form $F''_\Delta(\zeta_3)=C+D\zeta^2_3$, for some integers $C$ and $D$. Since $F''_\Delta(1)=C+D\not=0$, this implies that $C$ and $D$ cannot be both 0. Hence $F''_\Delta(\zeta_3)\not=0$. 
 
 (b) Suppose $p\equiv 2\pmod 3$. One has $F'_\Delta(-\zeta_3)$ is of the form $E+F\zeta_3$ for some integers $E,F$. Since $F'_\Delta(-1)=E+F\not=0$, this implies that $E$ and $F$ cannot be both 0. Hence $F'_\Delta(-\zeta_3)\not=0$. 

Now suppose that $p\equiv 1\pmod 3$. Then $p=12n+7$ for some integer $n$. One has $F_\Delta(-\zeta_3)=A\zeta_3+B\zeta_3^2$, where $A$ and $B$ are integers  and 

\[
\begin{aligned}
A&=-\sum_{\substack{a\equiv1 \pmod 6\\ 0<a<3p}} \left(\dfrac{3p}{a}\right)+\sum_{\stackrel{a\equiv 4 \pmod 6}{0<a<3p}}\left(\dfrac{3p}{a}\right)
=-\sum_{\substack{a\equiv1 \pmod 6\\ 0<a<3p}} \left(\dfrac{a}{p}\right)+\sum_{\stackrel{a\equiv 4 \pmod 6}{0<a<3p}}\left(\dfrac{a}{p}\right)\\
&=-\sum_{\substack{a\equiv 0,1,5 \pmod 6\\ 0<a<p}}\left(\dfrac{a}{p}\right)+ \sum_{\substack{a\equiv 2,3,4 \pmod 6\\ 0<a<p}}\left(\dfrac{a}{p}\right)=2\sum_{\substack{a\equiv 2 \pmod 6\\ 0<a<p}}\left(\dfrac{a}{p}\right)\\
&\equiv 2(2n+1)\equiv 2\pmod 4.
\end{aligned}
\]
Hence $A\not=0$ and $F_\Delta(-\zeta_3)\not=0$.
\end{proof}

\subsection{The Fekete polynomials $f_{\Delta}(x)$ and $g_{\Delta}(x)$}
We introduce the following definition. 
\begin{definition}
Let $\Delta =3p$ with $p \equiv 3 \pmod{4}.$ The Fekete polynomial $f_{\Delta}(x)$ is given by the following formula 
\[ f_{\Delta}(x)= \begin{cases} \dfrac{F_\Delta(x)}{x\Phi_1(x)^2 \Phi_2(x) \Phi_3(x)^2\Phi_6(x) \Phi_p(x) } &\mbox{if $p \equiv 2 \pmod{3}$} \\
\dfrac{F_\Delta(x)}{x\Phi_1(x)^2 \Phi_2(x) \Phi_3(x)\Phi_p(x)} & \mbox{if $p \equiv 1 \pmod{3}$}. \end{cases} \] 
\end{definition}

We have the following observation. 
\begin{prop}
$f_{\Delta}(x)$ is a reciprocal polynomial of degree 
\[ \deg(f_{\Delta})= \begin{cases} 2(p-5) &\mbox{if } p \equiv 2 \pmod{3} \\
2(p-3) & \mbox{if } p \equiv 1 \pmod{3}. \end{cases} \] 
\end{prop}
\begin{proof}
This follows from Lemma \ref{lem:reciprocal} and Proposition \ref{prop:reciprocal1}. 
\end{proof}

As before, we define $g_{\Delta}(x) \in \Z[x]$ to be the trace polynomial of $f_{\Delta}(x)$; namely 

\[ f_{\Delta}(x) = x^{\frac{\deg(f_{\Delta})}{2}} g_{\Delta}(x+\frac{1}{x}) .\]  

\section{Generalized Fekete polynomial with $\Delta=-3p$, $p \equiv 1 \pmod{4}$.}
\label{sec:-3p}
In this section, we discuss the case $\Delta=-3p$ with $p \equiv 1 \pmod{4}.$ We first have the following observation. 
\begin{prop}
$r_\Delta(\Phi_p)=r_\Delta(\Phi_3)=1$.
\end{prop}
\begin{proof}
This follows from Prop~\ref{prop:negative_big_divisor}.
\end{proof}

\begin{prop}
Let $p\equiv 1\mod 4$ be a prime number.
 The number $\zeta_3$ is a simple root of $F_\Delta$. In other words,
  $r_\Delta(\zeta_3)=1$.
\end{prop}
\begin{proof}
 One has 
$F'_\Delta(x)=\sum\limits_{k=1}^{p}(3k-2)\left(\dfrac{3p}{3k-2}\right)x^{3k-3}+\sum\limits_{k=1}^{p}(3k-1)\left(\dfrac{3p}{3k-1}\right)x^{3k-2}.$
Hence 
\[
F'_\Delta(\zeta_3)=\sum\limits_{k=1}^{p}(3k-2)\left(\dfrac{3p}{3k-2}\right)+\zeta_3\sum\limits_{k=1}^{p}(3k-1)\left(\dfrac{3p}{3k-1}\right)=A+B\zeta_3, 
\]
where
\[
\begin{aligned}
A&=\sum\limits_{k=1}^{p}(3k-2)\left(\dfrac{3p}{3k-2}\right)=\sum\limits_{k=1}^{p}(3k-2)\left(\dfrac{3k-2}{p}\right),\\
B&=\sum\limits_{k=1}^{p}(3k-1)\left(\dfrac{3p}{3k-1}\right)=\sum\limits_{k=1}^{p}(3k-1)\left(\dfrac{3k-1}{p}\right).
\end{aligned}
\]
Suppose that $p\equiv 2\pmod 3$. Then $p=12n+5$ for some integer $n$. Note that $3k-2$ is odd if and only if $k=2l+1$ for some $l$. In this case $3k-2=6l+1$. Hence 
$
A\equiv \sum\limits_{l=0}^{6n+2} \left(\dfrac{6l+1}{p}\right)
\pmod 2.
$
 Note also that the set $\{6l+1 \}_{l=0}^{6n+2}$ contains no element divisible by $p$. Hence the values of $\{\left(\frac{6l+1}{p} \right) \}, l=0,\ldots, 6n+3$, are  $\pm{1}$. Thus
 \[
 A\equiv \sum\limits_{l=0}^{6n+2}\left(\dfrac{6l+1}{p}\right)\equiv 6n+3\equiv 1\pmod 2.
 \]
 Therefore $A\not=0$ and $F'_\Delta(\zeta_3)\not=0$. 
 
 Now suppose that $p\equiv 1\pmod 3$. Then
 \[
\begin{aligned} A&=\sum_{\stackrel{a\equiv 1\pmod 3}{0<a<p}} a\left(\dfrac{a}{p}\right) + \sum_{\stackrel{a\equiv 1\pmod 3}{p<a<2p}} a\left(\dfrac{a}{p}\right)+\sum_{\stackrel{a\equiv 1\pmod 3}{2p<a<3p}} a\left(\dfrac{a}{p}\right)\\
&=\sum_{\stackrel{a\equiv 1\pmod 3}{0<a<p}} a\left(\dfrac{a}{p}\right) + \sum_{\stackrel{a\equiv 0\pmod 3}{0<a<p}} (p+a)\left(\dfrac{p+a}{p}\right)+\sum_{\stackrel{a\equiv 2\pmod 3}{0<a<p}} (2p+a)\left(\dfrac{2p+a}{p}\right)\\
&=\sum\limits_{a=1}^{p-1}a\left(\dfrac{a}{p}\right)+p\sum_{\stackrel{a\equiv 0\pmod 3}{0<a<p}} \left(\dfrac{a}{p}\right) +2p\sum_{\stackrel{a\equiv 2\pmod 3}{0<a<p}} \left(\dfrac{a}{p}\right).
\end{aligned}
 \]

 One has
 $
 \sum\limits_{a=1}^{p-1}a\left(\dfrac{a}{p}\right)=0 \quad (\text{see \cite[Equation 3]{[Girstmair]}}).
$
 Set 
 \[C=\sum\limits_{0<3a<p} \left(\dfrac{3a}{p}\right),\quad 
 D=\sum_{0<3a+1<p} \left(\dfrac{3a+1}{p}\right),\quad
 E=\sum_{0<3a+2<p} \left(\dfrac{3a+2}{p}\right).
 \]
 Then $C+D+E=0$ and
 $C=\sum\limits_{0<3a<p} \left(\dfrac{3a}{p}\right) =\sum_{0<3a<p} \left(\dfrac{p-3a}{p}\right)=D. 
 $
 Hence $E=-2C$ and $A=p(C+2E)=-3pC$. By \cite[Corollary 4.3]{[Berndt]}, 
 $
 C=\dfrac{1}{2}h(-3p).
 $
 Thus 
 $
 A=-\dfrac{3p}{2}h(-3p)\not=0.
 $
 \end{proof}

\subsection{The Fekete polynomials $f_{\Delta}(x)$ and $g_{\Delta}(x)$}

We introduce the following definition. 
\begin{definition}
Let $\Delta =-3p$ with $p \equiv 1 \pmod{4}.$ The Fekete polynomial $f_{\Delta}(x)$ is given by the following formula 
\[ f_{\Delta}(x)= \dfrac{-F_\Delta(x)}{x\Phi_1(x)  \Phi_3(x) \Phi_p(x)}.   \] 
\end{definition}

We have the following observation. 
\begin{prop}
$f_{\Delta}(x)$ is a reciprocal polynomial of degree 
$2(p-2)$.
\end{prop}
\begin{proof}
This follows from Lemma \ref{lem:reciprocal} and Proposition \ref{prop:reciprocal1}. 
\end{proof}

As before, we define $g_{\Delta}(x) \in \Z[x]$ to be the trace polynomial of $f_{\Delta}(x)$; namely 

\[ f_{\Delta}(x) = x^{\frac{\deg(f_{\Delta})}{2}} g_{\Delta}(x+\frac{1}{x}) .\]

Our data seems to suggest the following conjectures. 
\begin{conj} If $\Delta<0$ and odd then every irreducible cyclotomic factors of $F_\Delta(x)/x$  is of the form $\Phi_n(x)$ where $n$ is a proper divisor of $\Delta$. Furthermore, these factors have multiplicity one and the polynomial
\[
f_\Delta(x)=\dfrac{F_\Delta(x)}{x\prod_{n\mid \Delta}\phi_n(x)}
\]
is irreducible.
\end{conj}

\begin{conj} For each $n\geq 1$, let $r_\Delta(n)$ be the multiplicity of $\Phi_n(x)$ in  $F_\Delta(x)/x$, here multiplicity 0 is allowed. One has \[
F_\Delta(x)=\pm x \left( \prod_{n=1}^\infty \Phi_n(x)^{r_\Delta(n)}\right) f(x),
\]
for some monic {\bf irreducible} polynomial $f(x)\in \Z[x]$.
\end{conj}

\section{Galois theory for $f_{\Delta}(x)$ and $g_{\Delta}(x)$.}

We note that the Galois group of $g_{\Delta}(x)$ acts on the set of roots of $g_{\Delta}(x)$, so it is naturally a subgroup of $S_{h_{\Delta}}$ where $h_{\Delta} = \deg(g_{\Delta}(x))$. Additionally, the Galois group of $f_{\Delta}$ fits into the following exact sequence 

\[ 1 \to \Gal(\Q(f_{\Delta})/\Q (g_{\Delta})) \to \Gal(\Q(f_{\Delta})/\Q) \to  \Gal(\Q(g_{\Delta})/\Q) \to 1.\] 

By definition, $\Gal(\Q(f_{\Delta})/\Q (g_{\Delta}))$ is naturally a subgroup of $(\Z/2)^{h_{\Delta}}.$ We then conclude that the Galois group $\Gal(\Q(f_{\Delta})/\Q (g_{\Delta}))$ is a subgroup of the semi-direct product $(\Z/2\Z)^{h_{\Delta}}\rtimes S_{h_{\Delta}}$. Note that $(\Z/2\Z)^{h_{\Delta}}\rtimes S_{h_{\Delta}}$ is also naturally a subgroup of $S_{2 h_{\Delta}}$, the symmetric group on $2 h_{\Delta}$ letters (se \cite[Section 2]{[DDS]}). We have the following commutative diagram. 

\begin{lem} \label{lem:sgn}
\[\xymatrix@C+1pc{
         (\Z/2\Z)^{h_{\Delta}}\rtimes S_{h_{\Delta}} \ar[rd]^{\Sigma} \ar@{^{(}->}[r] &S_{2 h_{\Delta}}  \ar[d]^{\text{sgn}}  \\
         & \Z/2   \\
 }\]
Here $\text{sgn}$ is the signature map and $\Sigma$ is the following summation map 
\[ \Sigma (a_1, a_2, \ldots, a_{h_{\Delta}}, \sigma)  = \prod_{i=1}^{h_\Delta} a_i .\] 
\end{lem}
\begin{proof}
Let $n=h_\Delta$. We consider $(\Z/2\Z)^{n}\rtimes S_{n}$ as a subgroup of the symmetric group of $\{-n,-(n-1),\ldots,-1,1,\ldots,(n-1),n\}$ as follows: an element $((a_1,\ldots,a_n),\sigma)$ of $(\Z/2\Z)^{n}\rtimes S_{n}$ is considered as a permutation $\tilde{\sigma}$ which is defined by
\[
\begin{aligned}
\tilde{\sigma}(k)&=a_{\sigma(k)}\sigma(k), 1\leq k\leq n,\\
\tilde{\sigma}(-k)&=-a_{\sigma(k)}\sigma(k), 1\leq k\leq n.
\end{aligned}
\]
To prove the commutativity of the diagram, it is enough to consider the case $\sigma=(1\; 2)$. We consider the cyclic decomposition of $\tilde{\sigma}$. One has
\[
\begin{aligned}
1\stackrel{\tilde{\sigma}}{\mapsto} a_22\stackrel{\tilde{\sigma}}{\mapsto} a_1a_21\stackrel{\tilde{\sigma}}{\mapsto} a_12\stackrel{\tilde{\sigma}}{\mapsto} 1\\
 -1\stackrel{\tilde{\sigma}}{\mapsto} -a_22\stackrel{\tilde{\sigma}}{\mapsto} -a_1a_21\stackrel{\tilde{\sigma}}{\mapsto} -a_12\stackrel{\tilde{\sigma}}{\mapsto} -1\\
k\stackrel{\tilde{\sigma}}{\mapsto} a_k k \stackrel{\tilde{\sigma}}{\mapsto} k, \; \text{ for } 3\leq k\leq n\\
-k\stackrel{\tilde{\sigma}}{\mapsto} -a_k k \stackrel{\tilde{\sigma}}{\mapsto} -k, \; \text{ for } 3\leq k\leq n\\
\end{aligned}
\]
Note that for $k\geq 3$, if $a_k=1$ then $(k)$ is a $1$-cycle and its signature is 1 and   if $a_k=-1$ then $(k \; a_kk)$ is a $2$-cycle and its signature is -1. In any case, ${\rm sgn}((k \; a_kk))=a_k$ if $k\geq 3$.

If $a_1=a_2=1$ then one has two 2-cycles $(1\;2 )$ and $(-1\;-2)$ and the product of their signature is $1=a_1a_2$.
If $a_1=a_2=-1$ then one has two 2-cycles $(1\;-2 )$ and $(-1\;2)$ and   the product of their signature is $1=a_1a_2$.
If $a_1=1,a_2=-1$ then one has one 4-cycle $(1\;-2 \;-1\;2)$ and its signature is $-1=a_1a_2$.
If $a_1=-1,a_2=1$ then one has one 4-cycle $(1\;2 \;-1\;2)$ and its signature is $-1=a_1a_2$.

Thus ${\rm sgn}(\tilde{\sigma})=a_1a_2a_3\cdots a_n$, as desired.
\end{proof}
A direct corollary of this lemma is the following. 

\begin{cor} \label{cor:smaller_galois} 
Suppose that the discriminant $\text{disc}(f_{\Delta})$ of $f_{\Delta}$ is a non-zero perfect square. Then the Galois group of $f_{\Delta}$ is contained in the kernel of the map 
\[ \Sigma: (\Z/2\Z)^{h_{\Delta}}\rtimes S_{h_{\Delta}} \to \Z/2 .\] 
Note further that the kernel is isomorphic to to $\ker(\Sigma') \rtimes S_{h_{\Delta}}$, where $\Sigma'$ is the summation map 
\[ \Sigma': (\Z/2)^{h_{\Delta}} \to \Z/2 .\] 

\end{cor}

\begin{proof}
Let $u_1, u_2, \ldots, u_{2 h_{\Delta}}$ be all the roots of $f_{\Delta}.$ Let 
$ A = \sqrt{\text{disc}(f)} = \prod_{i<j} (u_j-u_i).$
For $\sigma \in \Gal(\Q(f_{\Delta})/\Q)$, we have $\sigma(A) = \text{sgn}(\sigma) A .$ By our assumption, $ A \in \Z$, so $\sigma$ must be an even permutation. By Lemma\ref{lem:sgn}, we conclude that $\sigma \in \ker(\Sigma)$. 
\end{proof}

\begin{rmk}
A quite interesting consequence of this corollary is that when $\text{dics}(f)$ is a perfect square, even though $f_{\Delta}$ is expected to be irreducible over $\Z$, it is reducible over $\F_q$ for all prime $q.$ In fact, if $f_{\Delta}$ is irreducible modulo $q$ then the Galois group of $f_{\Delta}$ must contain a $2 h_{\Delta}$-cycle. Since an $2h_{\Delta}$-cycle is an odd permutation, this contradicts the fact that all elements of the Galois group of $f_{\Delta}$ are even.  

\end{rmk}

Regarding the discriminant of $f_{\Delta}(x)$, we recall the following proposition. 
\begin{prop} (\cite[Proposition 4.4]{[MTT3]}) \label{prop:discriminant}
Let $f$ be a reciprocal polynomial of even degree $2n$ over a field of characteristics different from $2$. Let $g$ be the trace polynomial associated with $f$.  Let $s=(-1)^n f(1) f(-1)$. Then 
\[ {\rm disc}(f)= s \times {\rm disc}(g)^2. \]
Consequently, if $s$ is a perfect square then so is ${\rm disc}(f).$

\end{prop}

In \cite{[MTT3]}, we conjectured that the Galois group of $f_{\Delta}(x)$ and $g_{\Delta}(x)$ are as large as possible when $D= |\Delta|$ is a prime number (see \cite[Conjecture 4.9, Conjecture 4.13]{[MTT3]}); namely the Galois group of $g_{\Delta}$ is $S_{h_{\Delta}}$ and the Galois group of $f_{\Delta}$ is $(\Z/2\Z)^{h_{\Delta}}\rtimes S_{h_{\Delta}}$.  It is tempting to make a similar conjecture for all discriminant $\Delta$. Unfortunately, in our investigation,  we found that it is not always the case.  After some further investigation, we realize that this issue arises from the fact that, occasionally, there is an extra symmetry that puts some constraints on the Galois group of $f_{\Delta}$. More specifically, we discover some cases where the discriminant of $f_{\Delta}$ is a perfect square. By Corollary \ref{cor:smaller_galois}, the Galois group of $f_{\Delta}$ is contained in the kernel of the map $\Sigma.$ On the other hand, we also find that within the range of our data, it is always the case that the Galois group of $g_{\Delta}(x)$ is $S_{h_{\Delta}}.$ In this section, we provide some extensive data about the Galois groups of $f_{\Delta}(x)$ and $g_{\Delta}(x).$ As in \cite{[MTT3]}, when the Galois group of $f_{\Delta} =(\Z/2\Z)^{h_{\Delta}}\rtimes S_{h_{\Delta}}$, we can detect it by using \cite[Proposition 4.14]{[MTT3]}. Similarly, when the Galois group of $g_{\Delta}$ is $S_{h_{\Delta}}$, we can detect it by using \cite[Proposition 4.10]{[MTT3]}.

To deal with the case where the discriminant of $f$ is a square as discussed in Corollary \ref{cor:smaller_galois}, we need the following proposition which is a variant of \cite[Proposition 4.14]{[MTT3]}. We first introduce a lemma.

\begin{lem} \label{lem:Galois_computation_3}
Let $H$ be a subgroup of $(\Z/2)^n \rtimes S_n \subset S_{2n}$ such that 

\begin{enumerate}
    \item The natural projection map $H \to S_n$ is surjective.
    \item $H$ contains a product of a $2$-cycle and another $4$-cycle (these two cycles are disjoint). 
\end{enumerate}
Then $H$ contains $\ker(\Sigma') \rtimes S_n$ where $\Sigma'$ is the summation map 
\[ \Sigma': (\Z/2)^n \to \Z/2 .\] 
\end{lem}

\begin{proof}
We we recall that we consider $(\Z/2\Z)^{n}\rtimes S_{n}$ as a subgroup of the symmetric group of $\{-n,-(n-1),\ldots,-1,1,\ldots,(n-1),n\}$ by the following rules: an element $((a_1,\ldots,a_n),\sigma)$ of $(\Z/2\Z)^{n}\rtimes S_{n}$ is considered as a permutation $\tilde{\sigma}$ which is defined by
\begin{align*}
\tilde{\sigma}(k)=a_{\sigma(k)}\sigma(k), 1\leq k\leq n, \text{ and }
\tilde{\sigma}(-k)=-a_{\sigma(k)}\sigma(k), 1\leq k\leq n.
\end{align*} 
We can see that for all  $k$,
\begin{equation} \label{eq:sigma}
    \tilde{\sigma}(-k)= - \tilde{\sigma}(k).
\end{equation}
By the second assumption, we can assume that, up to an ordering, $H$ contains an element $\tilde{\sigma}=(a_1, a_2, \ldots, a_n, \sigma)$ of the form 
$ \tilde{\sigma} =(n, -n) (k_1, k_2, k_3, k_4),$
where $k_i$ are distinct and $|k_i| \in \{1, \ldots, n-1 \}$ for $1 \leq i \leq 4.$  By Equation \ref{eq:sigma} the 4-cycle $(-k_1, -k_2, -k_3, -k_4)$ is factor of $\tilde{\sigma}$ as well. Therefore, we must have 
\[ (k_1, k_2, k_3, k_4) = (-k_1, -k_2, -k_3, -k_4) .\]
We remark that $|k_1| \neq |k_2|$; otherwise $k_1 \to k_2 \to k_1$. Therefore, either $k_3=-k_1$ or $k_4=-k_1.$ The second case cannot happen because it will force $k_2= -k_2.$ So, $k_3=-k_1$ and consequently $k_4=-k_2.$ In other words, $\tilde{\sigma}$ is of the form 
\[ \tilde{\sigma} = (n, -n) (k_1, k_2, -k_1, -k_2) .\]
We then have 
$ \tilde{\sigma}^2 = (k_1, -k_1) (k_2, -k_2) \in H.$
Let $1 \leq i<j \leq n$. By the first assumption, we can find an element $\tilde{\tau}=(b_1, b_2, \ldots, b_n, \tau) \in H$ such that 
$ \tau = (k_1, i) (k_2 ,j).$ 
We can see that 
\[ \tilde{\tau} \tilde{\sigma}^2 \tilde{\tau}^{-1} = (\tilde{\tau}(k_1), \tilde{\tau}(-k_1))(\tilde{\tau}(k_2), \tilde{\tau}(-k_2) = (i, -i)(j,-j).\] 

This shows that the element $(c_1, c_2, \ldots, c_n, 1) \in H$ where $c_i=c_j=-1$ and $c_k=0$ if $k \not \in \{i,j \}.$ Consequently, $H$ contains $\ker(\Sigma') \times {1}.$ By our assumption, the natural map $H \to S_n$ is surjective, we can conclude that $H$ contains $\ker(\Sigma') \rtimes S_n.$
\end{proof}

We then have the following proposition. 
\begin{prop} \label{prop:galois_computation_4}
Let $f(x)$ be a monic reciprocal polynomial with integer coefficients of even degree $2n$. Let $g$ be the trace polynomial of $f$. Assume that 
\begin{enumerate}
    \item The Galois group of $g$ is $S_n.$
    \item There exists a prime number $q$ such that $f(x)$ has the following factorization in $\F_q(x)$
    \[ f(x)=p_2(x) p_4(x) h(x) ,\]
    where $p_2(x)$ is an irreducible polynomial of degree $2$, $p_4(x)$ is an irreducible polynomial of degree $4$, and $h(x)$ is a product of distinct irreducible polynomials of odd degrees.
\end{enumerate}
Then the Galois group of $f$ contains $\ker(\Sigma') \rtimes S_n.$ Furthermore 
\begin{enumerate}
    \item If the discriminant of $f$ is a perfect square then the Galois group of $f$ is $\ker(\Sigma') \rtimes S_n.$ 
    \item If the discriminant of $f$ is not a perfect square then the Galois group of $f$ is $(\Z/2)^n \rtimes S_n$.
\end{enumerate}
\end{prop}
\begin{proof}
Let $H$ be the Galois group of $f$.  The first condition and the second condition imply that $H$ satisfies the hypothesis described in Lemma \ref{lem:Galois_computation_3}. Therefore, $H$ contains $\ker(\Sigma') \rtimes S_n.$ If the discriminant of $f$ is a perfect square then Corollary \ref{cor:smaller_galois} implies the Galois group of $f$ is contained in $\ker(\Sigma') \rtimes S_n.$ Since we have inclusion on both ways, we conclude that $H=\ker(\Sigma') \rtimes S_n.$ If the discrimant of $f$ is not a prefect square then $H$ is strictly larger than $\ker(\Sigma') \rtimes S_n$. Therefore, it must be  $(\Z/2)^n \rtimes S_n$.
\end{proof}

\begin{cor}

Assume that $f$ satisfies the hypotheses described in Proposition \ref{prop:galois_computation_4}. Then $f$ is irreducible. 
\end{cor}
\begin{proof}
First, suppose that $f$ is a product of two reciprocal polynomials, say $f_1f_2$ of degree $m_1, m_2$ respectively. We can assume that $m_1, m_2$ are both even. In fact, we have 
$(-1)^{m_1} f_1(-1) = f_1(-1).$  If $m_1$ is odd, then $f(-1)=0$. Consequently, we can rewrite $f(x) = h_1(x) h_2(x)$, 
where $h_1(x)=\frac{h_1(x)}{x+1}$ and $h_2(x)= (x+1)f_2(x)$. We can see that both $h_1(x)$ and $h_2(x)$ are reciprocal polynomials of even degrees.

Now assume that both $m_1, m_2$ are even. Let $g(x), g_1(x), g_2(x)$ be the reciprocal polynomials of $f, f_1, f_2$ respectively. Then we have $g(x) =g_1(x)g_2(x).$
This is a contradiction because as the Galois group of $g$ is $S_n$, $g$ must be irreducible. Now, suppose that $f$ is not a product of two reciprocal polynomials. Then by \cite[Lemma 6]{[CC]}, either $f$ is irreducible over $\Z[x]$ or $f$ is of the form $f = a x^n h(x) h(\frac{1}{x})$, where $h \in \Q[x]$ and $a \in \Q^{\times}.$ However, this implies that the splitting field of $f$ is the same as the splitting field of $h$. Consequently, the Galois group of $f$ is at most $S_n.$ This contradicts the fact that the Galois group of $f$ is $\ker(\Sigma') \rtimes S_n$ as shown in Proposition \ref{prop:galois_computation_4}.
\end{proof}

We remark that by a similar argument, we have the following lemma. 

\begin{lem} \label{lem:Galois_computation_full}
Let $H$ be a subgroup of $(\Z/2)^n \rtimes S_n \subset S_{2n}$ such that 

\begin{enumerate}
    \item The natural projection map $H \to S_n$ is surjective,
    \item $H$ contains  a $2$-cycle. 
\end{enumerate}
Then $H=(\Z/2)^n \rtimes S_n$. 
\end{lem}
\begin{proof}
By the second assumption, we can assume that, up to an ordering, $H$ contains an element $\tilde{\sigma}$ of the form 
$ \tilde{\sigma} =(n, -n).$
Let $1 \leq i\leq n$. By the first assumption, we can find an element $\tilde{\tau}=(b_1, b_2, \ldots, b_n, \tau) \in H$ such that 
$ \tau = (n, i).$
We can see that 
$\tilde{\tau} \tilde{\sigma} \tilde{\tau}^{-1} =  (i, -i).$
This shows that the element $(c_1, c_2, \ldots, c_n, 1) \in H$, where $c_i=-1$ and $c_k=0$ if $k \not \in \{i\}.$ Consequently, $H$ contains  $(\Z/2)^n \times {1}.$ By our assumption, the natural map $H \to S_n$ is surjective, we can conclude that $H$ equals  $(\Z/2)^n \rtimes S_n.$
\end{proof}
As a consequence, we have another test to detect the Galois group of $f_{\Delta}$ when its Galois group is $(\Z/2)^{h_{\Delta}} \rtimes S_{h_{\Delta}}.$ 
\begin{prop} \label{prop:galois_computation_full}
Let $f(x)$ be a monic reciprocal polynomial with integer coefficients of even degree $2n$. Let $g$ be the trace polynomial of $f$. Assume that 
\begin{enumerate}
    \item The Galois group of $g$ is $S_n.$
    \item There exists a prime number $q$ such that $f(x)$ has the following factorization in $\F_q(x)$
    \[ f(x)=p_2(x) h(x) ,\]
    where $p_2(x)$ is an irreducible polynomial of degree $2$, and $h(x)$ is a product of distinct irreducible polynomials of odd degrees.
\end{enumerate}
Then the Galois group of $f$ is $(\Z/2)^n \rtimes S_n.$ 
\end{prop}

\begin{rmk}
We note that Proposition \cite[Proposition 4.14]{[MTT3]} and Proposition \ref{prop:galois_computation_full} look quite similar from a theoretical point of view. However, from a computational point of view, Proposition \ref{prop:galois_computation_full} is more effective since the degree of $g$ is half of the degree of $f$. For example, using the criterion described by \cite[Proposition 4.14]{[MTT3]}, we are able to compute the Galois group of $f_{\Delta}$ for $p$ at most $700$ (taking into account the running time of our algorithms). However, if we use the criterion mentioned in Proposition \ref{prop:galois_computation_full}, we are able to compute the Galois group for larger $p$ (currently for $p \leq 1000$). Furthermore, the running time is also faster.  

\end{rmk}

The data required to compute the Galois groups of $f_{\Delta}$ and $g_{\Delta}$ with $\Delta \in \{-4p, 4p, -3p, 3p \}$ is available on the Github repository \cite{[codes]}. Specifically, the data consists of the triple $(q_1, q_2, q_3)$ discussed \cite[Proposition 4.10]{[MTT3]}  whenever it exists (it seems that it always exists). It also contain the quadruple $(q_1, q_2, q_3, q_4)$ described by \cref{prop:galois_computation_4}  and \cref{prop:galois_computation_full}. Due to the computational complexity of our algorithms, we will restrict our search to the range 
\begin{equation} \label{eq:range} \max \{q_1, q_2, q_3, q_4 \} < 10^5. 
\end{equation} 

Below we described some concrete examples for the calculation of the Galois groups of $f_{\Delta}$ and $g_{\Delta}.$

\subsection{$\Delta =4p$ where $p \equiv 3 \pmod{4}$.}
We first search for the quadruple $(q_1, q_2, q_3, q_4)$ as required by Proposition \ref{prop:galois_computation_full}.
\begin{rmk}
For most primes $p \leq 547$, except $p=19$, we can always find the quadruple $(q_1, q_2, q_3, q_4)$. It turns out that Galois group of $f_{\Delta}$, where $\Delta=4\times 19$, is not $(\Z/2\Z)^{8}\rtimes S_{8}$. In this case 
\[f_\Delta(x)=x^{16} + x^{15} + 2x^{14} + 3x^{12} - x^{11} + 2x^{10} + 3x^8 + 2x^6 - x^5 + 3x^4 + 2x^2 + x + 1.\]
We also observe that 
 $f_{\Delta}(1) f_{\Delta}(-1) = 18^2.$
Consequently, by Lemma \ref{prop:discriminant}, we know that the discriminant of $f_{\Delta}$ is a perfect square. By Corollary \ref{cor:smaller_galois}, we conclude that the Galois group of $f_{\Delta}$ is a subgroup of the semi-direct product $ \ker(\Sigma') \rtimes S_8$ 
where $\Sigma'$ is the summation map 
\[ \Sigma': (\Z/2)^8 \to \Z/2 .\] 
To compute the Galois group of $f_{\Delta}$, we will use Proposition \ref{prop:galois_computation_4}. In fact, Table 3 shows that the Galois group of $g_{\Delta}$ is $S_8$. Additionally, at $q=227$, the factorization of $f_{\Delta}$ is 
\begin{align*}
&(x^2 + 153x + 1) (x^4 + 177x^3 + 43x^2 + 177x + 1) \times \\ &  (x^5 + 44x^4 + 148x^3 + 23x^2 + 196x + 207) (x^5 + 81x^4 + 101x^3 + 38x^2 + 134x + 34).
\end{align*} 
Therefore, by Proposition \ref{prop:galois_computation_4}, we conclude that the Galois group of $f_{\Delta}$ is exactly $ \ker(\Sigma') \rtimes S_8$. 
\end{rmk}


\subsection{The case $\Delta =-4p$ with $p \equiv 1 \pmod{4}$} In this case, we observe that for $p < 1000$, the Galois group of $f_{\Delta}(x)$ and $g_{\Delta}(x)$ are both as large as possible.

\subsection{The case $\Delta=3p$ with $p \equiv 3 \pmod{4}$}
From the data in \cite{[codes]} and Proposition \ref{prop:galois_computation_full}, we conclude that for $p<1000$, the Galois groups of $f_{\Delta}(x)$ and $g_{\Delta}(x)$ are both as large as possible

\subsection{The case $\Delta =-3p$ with $p \equiv 1 \pmod{4}$} 
It turns out that when $p \equiv 5 \pmod{8}$, there is an extra symmetry which puts some constraints on the Galois group of $f_{\Delta}.$ More precisely, we have the following proposition. 

\begin{prop} \label{prop:extra_symmetry} 
If $p\equiv 5\pmod 8$ and $\Delta=-3p$ then ${\rm disc}(f_\Delta)$ is a perfect square. 
\end{prop}
\begin{proof}
One has
$
xf_{\Delta}(x)= \dfrac{F_\Delta(x)}{\Phi_1(x)}\frac{1}{\Phi_3(x) \Phi_p(x)}.
$
Taking the limit when $x\to 1$, one obtains
\[
f_\Delta(1)=F'_\Delta(1)\dfrac{1}{\Phi_3(1)\Phi_p(1)}=\dfrac{F'_\Delta(1)}{D}=B_{1,\chi}.
\]
Here, the last equality follows from Equation (\ref{eq:F'(1)}).

One has
$
B_{1,\chi}=-h(-D),
$
where $h(-D)$ is the class number of the quadratic imaginary field $\Q(\sqrt{-D})$. This follows from the formula (\cite[Theorem 4.9 (i)]{[Washington]}),
$
L(1,\chi)=-\dfrac{\pi}{\sqrt{D}}B_{1,\chi}
$,
and the Dirichlet class number formula for a quadratic imaginary field (see e.g. \cite[Formula (15), page 49]{[Davenport]}), $
h(-D)=\dfrac{\sqrt{D}}{\pi}L(1,\chi).
$

On the other hand,
\[
f_\Delta(-1)=\dfrac{F_\Delta(-1)}{(-1)\cdot(-2)\cdot 1\cdot 1}=\dfrac{F_\Delta(-1)}{2}=\frac{\chi(2)h(-D)}{c},
\]
here the last equality follows from Equation (\ref{eq:F(-1)}). Recall that $c= 1$ if   $D\equiv 7\pmod 8$, and $c=1/3$
  if $ D\equiv 3\pmod 8$.
Hence, by Proposition \ref{prop:discriminant}
\[
{\rm disc}{(f_\Delta)}=(-1)^{p-2}f_\Delta(1)f_\Delta(-1){\rm disc}(g_\Delta)^2= \dfrac{\chi(2)}{c}h(-D)^2{\rm disc}(g_\Delta)^2.
\]
If $p\equiv 5\pmod 8$ then $D=3p\equiv 7\pmod 8$ and thus $c=1$ and  $\chi(2)=1$. Hence 
\[
{\rm disc}{(f_\Delta)}=h(-D)^2{\rm disc}(g_\Delta)^2,
\]
which is a perfect square.
\end{proof}

\begin{cor} \label{cor:p=5(8)}
Assume that $p \equiv 5 \pmod{8}$. Then the Galois group of $f_{\Delta}$ is a subgroup of of $ \ker(\Sigma') \rtimes S_{p-2}$. Here $\Sigma'$ is the summation map 
\[ \Sigma': (\Z/2)^{p-2 } \to \Z/2 .\] 
\end{cor}


\subsection{A conjecture} 
From these experimental data, it seems reasonable to make the following prediction. 
\begin{conj} \label{conj:main}
Let $\Delta \in \{4p, -4p, 3p, -3p \}$ and $h_{\Delta}$ be the degree of $g_{\Delta}$. Then
\begin{enumerate}
    \item If $\text{disc}(f_{\Delta})$ is not a perfect square then the Galois group of $f_{\Delta}$ is equal to  $(\Z/2\Z)^{h_{\Delta}}\rtimes S_{h_{\Delta}}$. 
    \item If $\text{disc}(f_{\Delta})$ is a perfect square then the Galois group of $f_{\Delta}$ is equal to $\ker(\Sigma') \rtimes (\Z/2)^{h_{\Delta}}$ where $\Sigma'$ is the summation map 
    \[ \Sigma' : (\Z/2)^{h_{\Delta}} \to \Z/2 .\] 
\end{enumerate}

In particular, the Galois group of $g_{\Delta}(x)$ is $S_{h_{\Delta}}.$
\end{conj}
\begin{rmk}
    Our experimental data in \cite{[codes]} shows that Conjecture \ref{conj:main} holds for $p \leq 1000.$
\end{rmk}

\section*{Acknowledgments}
We would like to thank Professor Franz Lemmermeyer for his encouragement to study this topic. The thread \cite{[Lemmermeyer]} initiated by him has been a critical inspiration for us. The second named author would like to thank William Stein for his help with the platform Cocalc where our computations are based. He also thanks to the organizers of the number theory seminars at UIC and Northwestern as well as the organizers of the CTNT 2022 Conference for providing him an opportunity to present parts of this work.  We would also like to thank Professors 
Andrew Granville, Peter Moree, and Ramin Takloo-Bighash for their interest and encouragement related to our work on Fekete polynomials. Last but not least, we are also grateful to the referee for their comments and valuable suggestions which we have used to improve our exposition.

\end{document}